\documentclass[11pt,letterpaper]{amsart}
\usepackage{graphicx} 
\usepackage{tikz-cd}
\usepackage{mathrsfs}
\usepackage{amsfonts}
\usepackage{amsmath}
\usepackage{amssymb}
\usepackage{amsthm}
\usepackage{enumitem}
\usepackage{comment}
\usepackage{eucal} 
\usepackage[all]{xy}

\theoremstyle{plain}
\swapnumbers
\newtheorem{lemma}{Lemma}[section]
\newtheorem{prop}[lemma]{Proposition}
\newtheorem{corollary}[lemma]{Corollary}
\newtheorem{theorem}[lemma]{Theorem}
\newtheorem{conjecture}[lemma]{Conjecture}

\theoremstyle{definition}
\newtheorem{definition}[lemma]{Definition}
\newtheorem{remark}[lemma]{Remark}
\newtheorem{construction}[lemma]{Construction}
\newtheorem{warning}[lemma]{Warning}
\newtheorem{example}[lemma]{Example}
\newtheorem{convention}[lemma]{Convention}

\newcommand{\End}{\operatorname{End}}
\newcommand{\Hom}{\operatorname{Hom}}
\newcommand{\Map}{\operatorname{Map}}

\newcommand{\supp}{\operatorname{supp}}
\newcommand{\im}{\operatorname{im}}
\newcommand{\vol}{\operatorname{vol}}

\newcommand{\codim}{\operatorname{codim}}

\newcommand\takuya[1]{{\color{blue} \sf $\infty$ Takuya: [#1]}}

\title{Weak Hodge theorem on piecewise-algebraic spaces}
\author{Aliakbar Hosseini and Takuya Murata}

\address{School of Mathematics, Institute for Research in Fundamental Sciences (IPM), Teheran, Iran}
\email{aliakbarman@gmail.com}

\address{School of Mathematics, Institute for Research in Fundamental Sciences (IPM), Teheran, Iran}
\email{takusi@gmail.com}

\date{\today}

\begin{document}

\setlist[enumerate,1]{label={(\roman*)}}

\begin{abstract} We prove a weak version of the classical Hodge theorem on \emph{piecewise-algebraic spaces}, a class of spaces introduced by Kontsevich and Soibelman in \cite{Kontsevich}. Precisely, we first prove the Poincar\'e lemma that computes singular cohomology as a variant of de Rham cohomology. Then, as a weak Hodge theorem, we naturally embed the singular cohomology into the space of harmonic forms, instead of establishing an isomorphism (which does not hold for those spaces). Our approach in the latter is classical: Sobolev space theory.

In addition, we give more detailed proofs for the claims in the appendix to \cite{Kontsevich}. This work is part of a program of extending arithmetic intersection theory to singular spaces. In particular, a type of currents in this singular setup is introduced.
\end{abstract}

\maketitle

\setcounter{section}{-1}
\setcounter{tocdepth}{1}

\tableofcontents

\section{Introduction}

The present paper extends the Hodge theorem on compact Riemannian manifolds to \emph{piecewise-algebraic spaces}, with a suitable formulation. The notion of a piecewise-algebraic space or a PA-space for short was introduced by Kontsevich and Soibelman in \cite{Kontsevich} and is a globalization of a real-semialgebraic set. Among PA spaces are algebraic varieties over real numbers and simplicial complexes.

A key fact about PA-spaces is that a PA-space admits an analog of smooth differential forms and that the Poincar\'e lemma holds for them. In particular, the singular cohomology of a PA-space can be computed as the PA-variant of the de Rham cohomology. After developing the general theory on PA-spaces, we prove a version of the Hodge theorem is also valid for PA-spaces; see Theorem \ref{Hodge theorem intro} below. (Some reformulation is needed since the usual Hodge theorem implies, for example, the Poincar\'e duality, which fails in general.)

The last section contains some additional materials; with that section, the paper also reproduces the whole Appendix of \cite{Kontsevich} with detailed proofs and minor fixes.

\subsection{Summary of results}

We summarize the results of this paper.

\S \ref{sec:PA-spaces} introduces PA-spaces. A \emph{PA-space} is, roughly, a locally compact space that admits a compact cover consisting of those isomorphic to compact semialgebraic subsets of $\mathbb{R}^m$ ($m$ depending on each compact set). Thus, compact PA-spaces play a role analogous to that of an open affine cover of a real algebraic variety (or more generally a scheme over real numbers). On a PA-space, using generic smoothness, we introduce the complex $\Omega_X$ of sheaves of forms on $X$ over the structure sheaf $\mathcal{O}_X$. The Poincar\'e lemma turns out to fail for it, as already observed in the original work \cite{Kontsevich}. Following the outline in op. cit., in \S \ref{sec:PA-forms}, we therefore introduce the complex $\Omega_{X, PA}$ of PA-forms, which is roughly generated by $\Omega_X$ and fiber integration, and then prove the Poincar\'e lemma for it. A proof in the compact case was already given in \cite{Hardt}; the proof in the noncompact case here seems new (although it is essentially the same as the classical one).

In \S \ref{sec:currents}, we introduce currents (to be precise, a variant of normal currents in geometric measure theory). A usual definition uses partial derivatives, which are not well-defined here. Thus, here, we adapt a minimalist definition so that the differential $d$ is well-defined and is continuous on currents. This result is used in the proof of the Hodge theorem but we will also need it to develop the Arakelov intersection theory on singular spaces in a future paper. We also prove the Poincar\'e lemma for currents, essentially by dualizing the Poincar\'e lemma for PA-forms.

Finally, in \S \ref{sec:Hodge theorem}, we prove the Hodge theorem, the main result of the paper. We write $\operatorname{H}(X; \mathbb{C})$ for the singular cohomology and $\operatorname{H}_{L^2_{loc}}(X)$ the cohomology given by locally square-integrable forms (PA-forms).

\begin{theorem}[Theorem \ref{L^2 Hodge} and Theorem \ref{singular Hodge theorem}]\label{Hodge theorem intro} Let $X$ be an oriented PA-space whose regular locus $X_{reg}$ is given a Riemannian metric that is uniformly continuous. Also, assume $\codim(X - X_{reg}) \ge 2$. Let $\Delta_d = d^* d + dd^*$ be the Laplacian on $X$, not just on $X_{reg}$.

Then
\begin{enumerate}
\item $\operatorname{ker}(\Delta_d) \overset{\sim}\to \operatorname{H}_{L^2_{loc}}(X), \, u \mapsto [u]$ is an isomorphism.
\item There are natural linear maps $\alpha, \beta$ such that 
$$\operatorname{H}(X; \mathbb{C}) \overset{\beta}\to \operatorname{ker}(\Delta_d) \overset{\alpha}\to \operatorname{H}(X; \mathbb{C})$$
is the identity.
\end{enumerate}
\end{theorem}

When $X$ is compact, the Poincar\'e duality holds for $\operatorname{H}_{L^2_{loc}}(X)$. On the other hand, the duality can fail on singular cohomology for some singular projective variety over complex numbers. Thus, the above $\alpha$ is generally not an isomorphism.

The proof here follows very closely the classical one, at least in the compact case. Thus, in particular, the essential tool is a Sobolev space. We prove the Hodge decomposition for a compact PA-space. The non-compact case is obtained more or less formally from the compact case. Unlike in other generalizations that consider $L^2$-forms on non-compact spaces, we only consider locally $L^2$-forms, which allows us to bypass some known difficulties in the non-compact case.


The present paper started as the authors' work on an extension of Arakelov intersection theory to singular spaces. In a sequel, we plan to develop the Arakelov intersection theory on singular arithmetic varieties in part using the Hodge theorem introduced in this paper.

\subsection{Possible extensions}

Most of the results in this paper should be true over a real closed field. One delicate issue is that the Heine-Borel theorem need not hold in such a setup. So, some care is needed in places where compactness is used. Similarly, functional analysis relies on the Cauchy completeness of the real numbers and that part has to be reworked. We note \cite{Ohmoto} essentially claims that the theory of PA-spaces can be done over real closed fields (and we tend to agree with it).

\subsection{Conventions and notations}
\begin{itemize}
\item Following Bourbaki, compact spaces, locally compact spaces and paracompact spaces are all assumed to be Hausdorff. However, we sometimes say Hausdorff for emphasis.
\item To be grammatically correct, we prefer to say a local-ringed space instead of a locally ringed space. This usage is after \cite{Knutson}.
\item Ring homomorphisms appearing in morphisms between ringed spaces are assumed to be $\Lambda$-linear for some fixed base ring $\Lambda$.
\item The term ``definable'' means real-semialgebraic as well as more general use for PA-spaces.
\item A manifold means a second-countable manifold with boundary.
\end{itemize}

Here are the notations used in the paper.
\begin{itemize}[label={}]
\item $\Sigma_X$, a triangulation on $X$ by simplicial complexes.
\item $\mathcal{O}_X$, the structure sheaf of a ringed space.
\item $\mathcal{O}_{X, \, pol}, \Omega_{X, \, pol}$, the sheaves of polynomial functions and polynomial forms.
\item $\Omega_{X, \, reg}$, the complex of sheaves of generically-smooth forms.
\item $\Omega_X \subset \Omega_{X, \, reg}$, the subcomplex generated by $\mathcal{O}_X, d(\mathcal{O}_X)$
\item $\Omega_{\mathcal{C}^k}$, $k$-times continuously differentiable forms on smooth PA-manifolds.
\item $\Omega_{chain-L^1}$, chain-wise integrable forms, Definition \ref{chain-wise integrable}.
\item $\Omega_{X, \, PA}$, PA-forms on a PA-space, Definition \ref{piecewise-algebraic space}.
\item $\Omega'$, currents, Definition \ref{def:currents}.
\item $\Omega_{L^2}$, the presheaf of $L^2$-forms.
\item $\Omega_{L^{2, \, loc}}$, the sheaf of locally $L^2$-forms.
\item $\Omega^*_{L^{2, \, loc}}$, the complex of sheaves determined by $\Omega_{L^{2, \, loc}}$.
\item $\Omega_{L^2, \,\mathcal{C}^{\infty}}$, the presheaf of $L^2$-forms whose restrictions to the regular locus are smooth.
\item $C(X; \Lambda)$, the space of \emph{nondegenerate} PA-chains on $X$ with coefficients in $\Lambda$.
\item $*$, the Hodge star operator or convolution
\item $\mu$, a fixed volume form.
\item $d^*$, the formal adjoint of $d$.
\item $\Delta_d = dd^* + d^* d$, the Laplacian.
\item $G$, the Green operator.
\item $W^s_{loc}$, the local Sobolev space of weight $s$.
\item $\widehat{\cdot}$, the Fourier transform, Construction \ref{Fourier transformation}.
\item $\operatorname{H}^*_{L^{2, \, loc}}(X)$, $L^{2, \, loc}$-cohomology, not necessarily $\operatorname{H}_{(2)}^*(X)$, usual $L^2$-cohomology.
\item $\operatorname{H}^*_{cpt}(X)$, the cohomology of the complex $(\Omega_{PA, \, cpt}(X), d)$.
\item $\mathbb{C}_{Po}$, the sheaf associated to $U \mapsto \operatorname{H}^n_{cpt}(U)^*$.
\end{itemize}

\section{Piecewise-algebraic spaces}\label{sec:PA-spaces}

In this section, we introduce a \emph{piecewise-algebraic space} or a \emph{PA-space}, the notion due to Kontsevich and Soibelman \cite[Appendix 8.1.]{Kontsevich}.

\subsection{PA-spaces} Let $\mathfrak F$ be the smallest set of subsets of $\mathbb{R}^n$ with the properties (1) it contains a set of the form $\{ x \in \mathbb{R}^n \mid f(x) > 0 \}$ where $f$ is a real polynomial function on $\mathbb{R}^n$ and (2) it is closed under finite union, finite intersection and complement (cf. \cite[Ch. II, Exercise 3.18.]{Hart}).

We call the elements of $\mathfrak{F}$ \emph{definable} (also called \emph{real semi-algebraic} in literature). In \cite[\S 8.1.]{Kontsevich}, it is called \emph{constructible}. However, we prefer to avoid ``constructible'' since that term also appears in stratification theory; e.g., in the discussion of perverse sheaves.

Following \cite[\S 8.1.]{Kontsevich}, given a subset $X$ in $\mathbb{R}^n$, the sheaf $\mathcal{O}_{X}$ of rings on $X$ is defined as follows: for each open subset $U \subset X$, $\mathcal{O}_X(U)$ consists of all complex-valued continuous functions $f$ on $U$ such that the graph of $f$ over each compact definable subset of $U$ is definable \cite[\S 8.1.]{Kontsevich}.

Or, equivalently by the closed graph theorem, saying a map between compact spaces is continuous if and only if it has closed graph, and the fact that $X$ is compactly generated (it is a metric space), $\mathcal{O}_X(U)$ consists of all complex-valued functions $f$ such that the graph of $f$ over each compact definable subset of $U$ is compact and definable. (When a space is compactly generated, a map to another space is continuous if and only if it is continuous on each compact subset; see \cite[Ch. 5., \S 1., Lemma]{May}.)

For an open subset $U \subset X$, $\mathcal{O}_X(U)$ contains all the polynomial functions on $U$. It, however, does not contain the exponential function when $X \subset \mathbb{R}$ since the graph of it is not definable over any compact definable subset of $U$. Also, if $f$ is definable and the derivative $f'$ exists, then $f'$ is definable.

We note that $(X, \mathcal{O}_X)$ is then a local-ringed space; a ringed space whose stalk at each point is a local ring. Indeed, if $x$ is a point in $X$, let $\mathfrak{m}_x$ be the set of all germs in $\mathcal{O}_{X, x}$ that vanish at $x$. For an $f$ in an ideal $I$, if $f \not\in \mathfrak{m}_x$, then $1/f$ is in $\mathcal{O}_{X, x}$, because the map $t \mapsto 1/t$ is continuous and definable away from $0$, and so $1 = f/f$ is in $I$.

Following \cite[\S 8.1. Definition 20.]{Kontsevich} and generalizing the definition of $\mathcal{O}_X$ above, we define a PA-space as follows. We say a family of subsets of a topological space $X$ is a \emph{compact cover} of $X$ if it consists of compact subsets, has union equal to $X$ and is locally-finite; i.e., each point has a neighborhood that intersects only finitely many members in the cover.

\begin{definition}\label{piecewise-algebraic space} A \emph{piecewise-algebraic space} or \emph{PA-space} $X$ is a locally-compact ringed space $(X, \mathcal{O}_X)$ satisfying the property that there are a compact cover $X_i$ of it as well as homeomorphisms
$$\iota_i : X_i \hookrightarrow \mathbb{R}^{n_i}$$
such that
\begin{enumerate}[label=(\arabic*)]
\item ${\iota_i(X_i \cap X_j)}$ is definable and each component of $\iota_j \circ \iota_i^{-1}|_{\iota_i(X_i \cap X_j)}$ is a global section of $\mathcal{O}_{\iota_i(X_i \cap X_j)}$,
\item for each open subset $U \subset X$, $\mathcal{O}_X(U)$ consists of all functions $f$ such that $f|_{U \cap X_i}$ is a section of $\mathcal{O}_{X_i}$, where $\mathcal{O}_{X_i}$ is defined using (1).
\end{enumerate}


The cover $X_i$ together with $\iota_i$ is called a \emph{defining cover}. A \emph{PA-map} is a map whose graph is a PA-space in a natural way \cite[\S 8.1.]{Kontsevich}\footnote{That is, given a $f : X \to Y$, choosing defining covers $X_i$ and $Y_j$ of $X, Y$ such that $f(X_i) \subset Y_j$, we require that the graphs of all those $f : X_i \to Y_j$ are compact definable subsets of Euclidean spaces in such a way they glue to a PA-space.} or equivalently a morphism of local-ringed spaces as seen below (Proposition \ref{prop:PA-map}).

\end{definition}

Note, by the closed graph theorem mentioned earlier, a PA-map is necessarily continuous.

\begin{remark} Property (2) of Definition \ref{piecewise-algebraic space} \emph{does not} mean $\mathcal{O}_X|_{X_i} \simeq \mathcal{O}_{X_i}$ for each $i$. The original definition of a PA-space is that $\mathcal{O}_X$ is ``locally isomorphic’’ to each $\mathcal{O}_{X_i}$, where the meaning of locally isomorphic is unclear to us. We believe our definition is a correct interpretation of the original definition (see also Remark \ref{compact topology} below for an alternative but equivalent definition).
\end{remark}

\begin{remark}[bijective PA-maps]\label{bijective PA-maps} A bijective PA-map is necessarily an isomorphism. Indeed, the inverse is a PA-space since the graph of it is the same as the original map.
\end{remark}

We note that $X$ is in fact local-ringed; i.e., each stalk of $\mathcal{O}_X$ is a local ring for the same reason as before. Also,

\begin{prop}\label{prop:paracompact} Each PA-space is paracompact.
\end{prop}
\begin{proof} Let $X$ be a PA-space. Since $X$ is locally compact, it is regular (\cite[Ch. I., § 9., No. 7., Proposition 9.]{bourbaki}). Thus, by Michael's theorem (\cite[Theorem 20.7.]{Willard}), it is enough to show each open cover $\mathcal{U}$ admits a locally finite refinement not necessarily of open sets.

Let $X_i$ be some defining cover of $X$, which is locally finite by definition. Then, for each $x$ in $X$, we can find a neighborhood $V_x$ of it such that only finitely many $X_i$'s intersect $V_x$. Each $x$ is contained in some $U \in \mathcal{U}$. Then, shirking $V_x$, we can assume $x \in V_x \subset U$. For each $i$, since $X_i$ is compact, among $V_x$'s, we can find finitely many $V_{i, j}, 1 \le j \le m_i$ that cover $X_i$. Then $\{ V_{i, j} \cap X_i \mid i, 1 \le j \le m_i \}$ is a locally finite refinement, since the cover $X_i$ is locally finite.
\end{proof}



It is clear that a PA-map (Definition \ref{piecewise-algebraic space}) induces a local-ringed-space morphism; i.e., a pair
$$(f : X \to Y, \, \varphi : \mathcal{O}_Y \to f_* \mathcal{O}_X)$$
consisting of a continuous map and a sheaf morphism that induces local homomorphisms between stalks. The converse holds by the following:

\begin{prop}\label{prop:PA-map} Each local-ringed-space morphism between PA-spaces is induced by a PA-map, where, by convention, ring homomorphisms are assumed to be $\mathbb{C}$-linear.

In particular, an isomorphism of PA-spaces as local-ringed spaces is the same thing as a bijective PA-map (Remark \ref{bijective PA-maps})
\end{prop}
\begin{proof} Given a continuous map $f : X \to Y$ with $\varphi : \mathcal{O}_Y \to f_* \mathcal{O}_X$, we need to show $\varphi = f^*$ and $f$ is a PA-map.

First we shall show $\varphi = f^*$; that is, for an open subset $V \subset Y$ and $g$ in $\mathcal{O}_Y(V)$, we show $\varphi(g) = g \circ f$. Suppose $f(a) = b$. Then since $\varphi : \mathcal{O}_{Y, b} \to\mathcal{O}_{X, a} $ is local, $\varphi(\mathfrak{m}_b) \subset \mathfrak{m}_a.$ We have $g - g(b) \in \mathfrak{m}_b$ trivially 
and so $$\varphi(g) - g(b) \in \mathfrak{m}_a$$
since $\varphi$ is $\mathbb{C}$-linear. That is, $\varphi(g)(a) = g(b)$, which proves the claim. The claim also shows that $f$ is a PA-map since it implies that locally on compact definable sets, each coordinate component of $f$ is definable.
\end{proof}

We frequently use the following:

\begin{lemma}\label{fiber product} The fiber product of two PA-maps exists; i.e., if $f : Y \to X$, $g : Z \to X$ are PA-maps, then we have the PA-space $Y \times_X Z$ and PA-maps $p : Y \times_X Z \to Y$, $q : Y \times_X Z \to Z$ that satisfy the universal property of a fiber product.
\end{lemma}
\begin{proof} In general, in a category, fiber products can be constructed from finite products and equalizers and it is clear that finite products and equalizers exist; indeed, the product $X \times Y$ exists since $\{ X_i \times Y_j \mid i, j \}$ is a compact cover if $X_i$ and $Y_i$ are compact covers of $X \times Y$.
\end{proof}

\begin{remark}[A PA-space as a sheaf]\label{compact topology} Definition \ref{piecewise-algebraic space} can be restated in terms of sheaves in the category theory sense. Indeed, let $P$ denote the category where objects are compact definable subsets of Euclidean spaces and morphisms are maps whose graphs are objects in the category. It has finite fiber products by Lemma \ref{fiber product}.

We can then consider a pre-sheaf of sets on $P$; i.e., a contravariant functor from $P$ to the category of sets. By a \emph{covering of an object $Y$ in $P$}, we mean a finite surjective family of morphisms $f_i : Y_i \to Y, \, i \in I$, where ``surjective family'' means $Y = \cup_i \im(f_i)$.

Then a pre-sheaf $F$ on $P$ is a \textit{sheaf} if it is a sheaf with respect to these coverings; that is, we have the equalizer
$$F(Y) \to \prod_i F(Y_i) \rightrightarrows \prod_{i, j} F(Y_i \times_Y Y_j)$$
for a covering $Y_i \to Y, \, i \in I$. For example, if $X$ is a PA-space, then $\Hom(-, X)$ is a sheaf on $P$, where $\Hom(Y, X) = $ the set of all PA-maps $Y \to X$. Moreover, $F = \Hom(-, X)$ satisfies the property: there is an epimorphism of presheaves
$$\sqcup_i h_{X_i} \to F$$
for some objects $X_i$ in $P$, where $h_Y = \Hom_P(-, Y)$, such that 
\begin{enumerate}[label=(\arabic*)]
\item For each $i$ and each object $T$ in $P$, $h_{X_i} \times_F h_T$ is representable, where a fiber product is that of presheaves.
\item The family $\{ X_i \mid i \}$ is locally finite; i.e., for each $T$ in $P$, $h_{X_i} \times_F h_T$ is nonempty only for finitely many $i$'s. (This definition of ``locally finite'' is equivalent to the usual one because of the ``locally compact'' assumption.)
\end{enumerate}

Conversely, if $F$ is a sheaf on $P$ satisfying the above property, then for each $i, j$, (1) gives an object $Z$ in $P$ such that $h_Z = h_{X_i} \times_F h_{X_j}$. By Yoneda's lemma, the projections $p : Z \to X_i, \, q : Z \to X_j$ are both in $P$ and so $X_i$'s glue to a PA-space $X$ such that $\Hom(-, X) = F$ as they agree on a covering.

(The above definition is analogous to the definition of algebraic spaces in algebraic geometry; e.g., \cite[Ch. 2., Definition 1.1.]{Knutson}.)
\end{remark}

In the language of scheme theory, a PA-space is an analog of a scheme of locally of finite type, rather than a scheme of finite type. Thus, we have the following, which is important for the authors of this paper.

\begin{example} A reduced scheme that is separated and locally of finite type over $\mathbb{C}$ and that has a countable affine cover is a PA-space in the classical topology. Indeed, such a space is clearly paracompact and locally compact and so, to see the claim, it is enough to consider affine such scheme $X$; i.e., $X \subset \mathbb{C}^n$ is a Zariski-closed subset. Since $X$ can be covered by $X \cap \overline{B_i}$, $\overline{B_i}$ closed balls, we have that $X$ is a PA-space.
\end{example}

Thus, a complex algebraic variety is in particular a PA-space. In addition, a cube or a simplex is a PA-space.

\begin{example} A finitely generated convex cone in $\mathbb{R}^n$ is an example of a non-compact PA-space; indeed, it is a finite intersection of half-spaces.

More generally, if $\Sigma$ is a toric fan in $\mathbb{R}^n$,\footnote{A toric fan is a set of finitely generated rational strictly convex cones that satisfies (1) each face of a cone in the fan is in the fan and (2) an intersection of two cones in the fan is a common face of each cone.} then the support $|\Sigma| := \bigcup_{\sigma \in \Sigma} \sigma$ of it is a PA-space, even if $\Sigma$ is not finite, since it is obtained by gluing PA-spaces.
\end{example}

\begin{example}[\cite{Robson} Example 3.1.] Let $U_1 = (0, 1)^2$ and $U_2$ the union of the open unit disk and the point $(0, 1)$. Let $X = U_1 \cup U_2$ with the topology that a subset $U$ is open if and only if $U \cap U_i$ is open for $i = 1, 2$. Let $A = \{ (x, y) \mid x^2 + y^2 = 1, x > 0, y > 0 \}$. Then $A$ is closed in $X$ since $A \cap U_1$ is closed and $A \cap U_2$ is empty (thus closed). Now, $A$ and $(0, 1)$ cannot be separated by disjoint open sets in $X$; thus, $X$ is not regular and in particular not paracompact (and not a PA-space).

A theorem of Robson (op. cit., Theorem 1.) says that a regular topological space obtained by gluing finitely many definable subsets of Euclidean spaces as open sets is definably homeomorphic to a definable subset of a Euclidean space. Thus, such a space is a PA-space.
\end{example}

\subsection{Topological results}
A basic question is whether a PA-space can be (simplicially) triangulated or not; i.e., whether a PA-space is a \emph{pl-space} \cite[Ch. I., \S 1.3.]{Borel} or not. We do not know the answer to that question, but we still have the following:

\begin{prop} A compact PA-space can be triangulated.
\end{prop}
\begin{proof} Let $X_i$ be a defining cover of $X$. Since they are locally finite, it has to be a finite cover; denote them by $X_1, \dots, X_m$. We shall prove the proposition by induction on $m$. The basic case $m = 1$ is clear by \cite[Th\'eor\`em 9.2.1.]{Bochnak}. Next, by inductive hypothesis, $X' = \bigcup_{i=1}^{m-1} X_i$ has a triangulation; in particular, $X'$, the geometrization of the triangulation, is embedded into a Euclidean space as a compact definable subset. Hence, $X'$ can be used among a defining cover of $X$ and it is thus enough to prove the case $m = 2$.

Let $Y = X_1 \cap X_2$. By \cite[Th\'eor\`em 9.2.1.]{Bochnak}, we find a triangulation $\Sigma_{X_1}$ on $X_1$ that restricts to triangulation on $Y$. Then, by the same theorem, we find $\Sigma_{X_2}$ that restricts to triangulations on $Y$ such that $\Sigma_{X_2}|_Y$ refines $\Sigma_{X_1}|_Y$. Refining $\Sigma_{X_2}$, we can assume the refinement $\Sigma_{X_2}|_Y \to \Sigma_{X_1}|_Y$ is a sequence of barycentric subdivisions (i.e., refinements obtained by inserting new vertices). If $\Sigma_{X_1}|_Y(w) \to \Sigma_{X_1}|_Y$ is a barycentric subdivision given by a new vertex $w$, then the subdivision $\Sigma_{X_1}(w) \to \Sigma_{X_1}$ restricts to $\Sigma_{X_1}|_Y(w) \to \Sigma_{X_1}|_Y$. Hence, we can find a refinement $\Sigma'_{X_1} \to \Sigma_{X_1}$ that restricts to the refinement $\Sigma_{X_2}|_Y \to \Sigma_{X_1}|_Y$. Then $\Sigma_X = \Sigma'_{X_1} \cup \Sigma_{X_2}$ is a triangulation of $X$.
\end{proof}


%

Just as for $\mathbb{R}^n$, given a PA-space $X$, let $\mathfrak{F}$ be the smallest set that contains (1) $\{ x \in X \mid f(x) > 0 \}$, $f$ real-valued functions in $\mathcal{O}(X)$ and (2) it is closed under finite union, finite intersection and complement. Then, by \emph{definable} (or \emph{constructible}) subsets of $X$, we shall mean the members of this set $\mathfrak{F}$. The next corollary implies this definition coincides with the previous one in the compact case.

\begin{corollary}[\cite{Kontsevich} just after \S 8.1. Definition 20.] If $X' \subset X$ is a compact definable subset, then $X'$ is isomorphic to a (compact) definable subset of $\mathbb{R}^n$.

In particular, a compact PA-space is isomorphic to a definable subset of a Euclidean space.
\end{corollary}

\begin{corollary} A compact PA-space has a unique pl-space structure.
\end{corollary}
\begin{proof} The uniqueness is by the Hauptvermutung \cite{Shiota1984}.
\end{proof}

\begin{corollary} Each open definable subset $U$ of a PA-space with compact closure can be triangulated by relatively-open simplexes.
\end{corollary}
\begin{proof} Note $\overline{U}$ is also definable. Thus, by the preceding corollary, $\overline{U}$ has a pl-space structure. In general, an open subset of a pl-space has an induced pl-space structure \cite[Ch. I., \S 1.3.]{Borel}.
\end{proof}

\begin{corollary}\label{locally contractible} A PA-space is locally contractible in the sense: at each point, there is a neighborhood base consisting of contractible open subsets.
\end{corollary}
\begin{proof} Let $X$ be a PA-space and $x$ a point in $X$. Since $X$ is locally compact, $x$ has a neighborhood $U$ with compact closure. Since $U$ is a union of definable subsets, shrinking $U$, we can assume $U$ is definable. Then $U$ is triangulated by the preceding corollary and so locally contractible.
\end{proof}


Recall that, given a sheaf $F$ on a topological space $X$, $F$ can be identified with the sheafification of $F$ as a presheaf. Explicitly, that means: for each open subset $U$, $F(U)$ is the set of all continuous sections
$$U \to E(F)$$
to the topological space $E(F)$ called the \emph{total space} of $F$ or more historically \emph{espace \'etal\'e} of the presheaf $F$ \cite[Ch. II, Exercise 1.13.]{Hart}. Viewing $F$ this way, we can consider a section of $F$ over a closed subset (in fact any subset) of $X$. A sheaf is then called \emph{soft} if each section over a closed subset extends to a global section \cite[Ch. II, \S 3.4.]{Godement}. For example, a flasque sheaf on a paracompact space is soft [loc. cit.].

We note:

\begin{remark}\label{section extends} If $X$ is a PA-space and $F$ a sheaf on it, then every section of $F$ over a subset $S \subset X$ extends to a section over a neighborhood of $S$, by \cite[Ch. II, Th\'eor\`em 3.3.1.]{Godement}.

(Some authors define a soft sheaf to be a sheaf such that for each section $s$ over a neighborhood of a closed subset $A$, the restriction $s|_A$ extends to a global section. For PA-spaces, the definition \`a la Godement thus coincides with that definition.)
\end{remark}

The next result is a key for constructing a partition of unity.

\begin{lemma}[bump function]\label{lem:nontrivial func} For each point $x$ in a PA-space $X$ and a neighborhood $U$ of $x$, there exists a nonnegative function $\psi$ in $\mathcal{O}(X)$ such that $\supp(\psi)$ is a compact subset of $U$ and $\psi(x) > 0$.
\end{lemma}
\begin{proof} If $X \subset \mathbb{R}^n$ is a compact definable subset, then the assertion holds by \cite[Proposition 2.7.4.]{Bochnak}.

In general, shrinking $U$, assume $\overline{U}$ is compact. Among the members of some defining cover of $X$, there are at most finitely many $X_i$'s containing $x$. Choose $\varphi_i$ in $\mathcal{O}_{X_i}(X_i)$ whose support lies in $U \cap X_i$ and such that $\varphi_i(x) > 0$. By Remark \ref{section extends}, $\varphi_i$ extends to a section $\psi_i$ of $\mathcal{O}_X$ over some neighborhood of $X_i$. Then the function $\psi = \prod_i \psi_i$ has the required property, since $\supp(\psi) \subset \cup_i (U \cap X_i) = U$.
\end{proof}

The existence of a partition of unity is now a formal consequence.

\begin{prop}[partition of unity]\label{partition of unity} Let $X$ be a PA-space and $U_i$ an open cover of $X$.

Then there exist nonnegative functions $\lambda_i$ in $\mathcal{O}_X(X)$ such that $\supp(\lambda_i) \subset U_i$, the family $\{ \supp(\lambda_i) \mid i \}$ is locally finite and $1 = \sum_i \lambda_i$.

In particular, given a closed subset $A \subset X$ and a neighborhood $U$ of $A$, there exists a function $\psi$ in $\mathcal{O}_X(X)$, called a cut-off function, such that $\psi = 1$ on $A$ and $\supp(\psi) \subset U$.
\end{prop}
\begin{proof} The proof is by essentially repeating the one in differential topology; e.g., \cite[Theorem 2.1.]{Hirsch}.

First note that if we can find a partition of unity $\{\mu_v\}_{v\in \mathcal{V}}$ for a refinement $\mathcal{V}$ of the given cover $\mathcal{U}$, then we can find a partition of unity of covering $\mathcal{U}$. Take $\{\mu_v\}_{v \in \mathcal{V}}$ to be a partition of unity for the refined covering $\mathcal{V}$, pick a choice function $c : \mathcal{V} \rightarrow \mathcal{U}$ and define $J_u = \{v : c(v) = u \}$. Then $\lambda_u = \sum_{v \in c^{-1}(u)} \mu_v$ are a partition of unity clearly. Hence, since $X$ is paracompact by Proposition \ref{prop:paracompact}, replacing the given open cover $\mathcal{U}$ by a locally finite refinement, we assume the cover $\mathcal{U}$ is locally finite. Since $X$ is locally compact, we can also assume each $U \in \mathcal{U}$ has compact closure.

Since $X$ is paracompact, we can then find open subsets $V_i \subset U_i$ such that (1) the $V_i$'s are an open cover of $X$ and (2) $\overline{V_i} \subset U_i$ (e.g., \cite[Theorems A.1., A.2.]{Hirsch}). Note each $\overline{V_i}$ is compact.

Now, for each $i$, by the preceding Lemma \ref{lem:nontrivial func}, we can find finitely many $\psi_{i, j}$'s such that $\{ \psi_{i, j} > 0 \}$ covers $\overline{V_i}$ and $\supp(\psi_{i, j}) \subset U_i$. Then the functions $\lambda_i = (\sum_j \psi_{i, j}) / (\sum_{i, j} \psi_{i, j})$ are the required partition of unity.

\end{proof}

\begin{corollary}\label{Omega soft} Every $\mathcal{O}_X$-module is soft.
\end{corollary}
\begin{proof} Let $s$ be a section of an $\mathcal{O}_X$-module $F$ over a closed subset $A$. By Remark \ref{section extends}, $s = t|_A$ for some section $t$ over a neighborhood $U$ of $A$. Then $\psi t$ is a global section that extends $s$ when $\psi$ is in $\mathcal{O}(X)$ such that $\supp(\psi) \subset U$ and $\psi = 1$ on $A$.
\end{proof}

The next lemma is well-known for topological manifolds (see e.g., \cite[\S 5.31.]{Warner}).

\begin{lemma}\label{singular cohomology} If $X$ is a PA-space, then the singular cohomology of $X$ with coefficients in $\mathbb{C}$ is the same as the sheaf cohomology of $X$ with values in the constant sheaf $\mathbb C$.
\end{lemma}
\begin{proof} First we have some general discussion. We know that a soft sheaf $F$ is $\Gamma(X, -)$-acyclic; i.e., $\operatorname{H}^i(X, F) = 0, i > 0$ \cite[Ch. II, Th\'eor\`em 4.4.3. (b)]{Godement}. Also, it is known and is easy to see that sheaf cohomology can be computed using a resolution by $\Gamma(X, -)$-acyclic sheaves instead of an injective resolution \cite[Ch. III, Remark 2.5.1.]{Hart}.

Let $S_p(X)$ be the free abelian group generated by $\Map(\Delta^p, X)$, the space of continuous $p$-simplexes, and $CS^p(X) = \Hom_{\mathbb{Z}}(S_p(X), \mathbb{C})$. They come with the boundary and coboundary operators respectively. The singular cohomology of $X$ is, by definition, the cohomology of the cochain complex $CS(X)$.

We shall now follow \cite[Exemple 3.9.1.]{Godement}. Let $\mathcal{S}^p$ be the sheafification of $U \mapsto CS^p(U)$, called the sheaf of localized singular $p$-cochains. It is a soft sheaf according to loc. cit. Also, by loc. cit., the natural map $CS^p(X) \to \Gamma(X, \mathcal{S}^p)$ is surjective and thus we can write:
$$0 \to K^p \to CS^p(X) \to \Gamma(X, \mathcal{S}^p) \to 0.$$
Now, by a calculation in \cite[\S 5.32.]{Warner}, we see that $\operatorname{H}^p(K^{\bullet})$ is zero for all $p$ and, consequently, the singular cohomology of $X$ is the cohomology of the complex $\Gamma(X, \mathcal{S}^{\bullet})$.

It remains to show $\mathcal{S}^{\bullet}$ resolves $\mathbb{C}$. By Proposition \ref{locally contractible}, $X$ is locally contractible. By the homotopy invariance of singular cohomology (\cite[Ch 18., \S 1.]{May}), it follows that higher singular cohomology of some neighborhood of each point in $X$ vanishes. It thus tells the exactness except $\ker(d : \mathcal{S}^0 \to \mathcal{S}^1) = \mathbb{C}$. But that latter can be seen directly.
\end{proof}

Finally, we record some results on PA-homotopies. First, we note the following lemma. Recall that a map $A \hookrightarrow X$ is a \emph{cofibration} if given $f : X \to Z$, each homotopy $g_t : A \to Z$ such that $g_0 = f|_A$ extends to a homotopy $h_t : X \to Z$ such that $h_0 = f$.

By a \emph{reasonable topological space}, we mean a compactly generated weak Hausdorff space (\cite[Ch. 5.]{May}).

\begin{lemma}\label{compact subspace cofibration} We have:
\begin{enumerate}
\item A cofibration $i: A \hookrightarrow X$ is a homotopy equivalence if and only if there is a homotopy $f_t : X \to X$ such that $f_1$ is the identity, $f_t \circ i$ is the identity and $f_0 = i \circ r$ for some $r : X \to A$ with $r \circ i = \operatorname{id}$; i.e., $r$ is a strong deformation retraction.
\item If $K \subset X$ is a compact subspace of a PA-space, then $K \hookrightarrow X$ is a cofibration; in fact,  each neighborhood of $K$ contains a smaller neighborhood $K_{\epsilon}$ such that $K \hookrightarrow K_{\epsilon}$ is both a cofibration and a homotopy equivalence.
\end{enumerate}

Both of the above statements hold in the category of reasonable topological spaces as well as the category of PA-spaces.
\end{lemma}
\begin{proof} (i) holds by \cite[Ch 6., \S 5.]{May}. Indeed, by the loc. cit., it is a homotopy equivalence under $A$, meaning $i \circ r \sim \operatorname{id}_X$ and $r \circ i \sim \operatorname{id}_A$ in such a way the homotopies here are the identity on $A$. Note $r \circ i = \operatorname{id}_A$.

(ii) is clear because of the existence of a triangulation. Precisely, we choose a triangulation on $K$ and also a compatible triangulation on a relatively compact neighborhood $U$ of $K$. For a CW complex, we know the inclusion of a subcomplex is a cofibration and the same is also valid for simplicial complexes. (Note the proof of the theorem in \cite[Ch 6., \S 4.]{May} may not go through with PA-maps, but here we do not use that result.)
\end{proof}

In passing, we mention

\begin{conjecture} For each closed subspace $Y \subset X$, the natural inclusion $Y \hookrightarrow X$ is a cofibration, both in the category of reasonable topological spaces and in the category of PA-spaces.
\end{conjecture}

A tricky issue here is that we are not assuming PA-spaces are of finite type, meaning they admit finite open covers consisting of those isomorphic to closed subspaces in Euclidean spaces. For PA-spaces of finite type (e.g., a complex algebraic variety), the conjecture is likely clear; e.g., by \cite{Murata_Bossinger25}. Our guess for a proof of the general case is to embed a PA-space into a Hilbert space.

Next, given an open cover $\alpha$ of a topological space $Y$, we say two continuous maps $f, g : X \to Y$ are \emph{$\alpha$-near} if $\{ f^{-1}(U) \cap g^{-1}(U) \mid U \in \alpha \}$ covers $X$. Also, an \emph{$\alpha$-homotopy} $h_t : X \to Y$ is a homotopy such that $\{ \cap_{t \in I} h_t^{-1}(U) \mid U \in \alpha \}$ covers $X$.

\begin{lemma}[cf. \cite{Hanner} Theorem 4.1.]\label{approximate homotopy} Given a compact PA-space $Y$ and an open cover $\alpha$ of $Y$, there exists a refinement $\beta$ of it with the following property: for two continuous maps $f, g : X \to Y$ from a compact $PA$-space that are $\beta$-near and $A \subset X$ a closed subspace, each $\beta$-homotopy $f|_A \sim g|_A$ extends to an $\alpha$-homotopy $f \sim g$.

Moreover, the above also holds in the category of PA-spaces.
\end{lemma}
\begin{proof} The only question is whether the original proof of the theorem cited above goes through with PA-maps. But that is clear except perhaps the following claim:

\textbf{Claim}: A homotopy $h_t : f|_A \sim g|_A$ extends to a homotopy $f|_U \sim g|_U$ for some neighborhood $U$ of $A$.

Here, the claim follows from the preceding Lemma \ref{compact subspace cofibration}. Indeed, let $r : U \to A$ be a \emph{strong} deformation retract of $i : A \hookrightarrow U$ for some neighborhood $U$ of $A$. Then we have:
$$f|_U \sim f|_U \circ i \circ r = f|_A \circ r \overset{h_t \circ r}\sim g|_A \circ r \sim g|_U$$
where $f|_U \sim f|_U \circ i \circ r$ is the identity on $A$ because of the strong-ness of $r$. Hence, it is an extension of $h_t$.
\end{proof}

The next proposition confirms our intuitive expectations on connectedness.

\begin{prop} For PA-spaces $X, Y$, each connected component of $\Map(X, Y)$ is path-connected with respect to the Whitney topology, a topology where the basic open subsets are
$$\{ f \mid \Gamma_f \subset U \}$$
for open subsets $U \subset X \times Y$ and $\Gamma_f$ the graph of $f$.

Also, a path between PA-maps can be taken to be a PA-path.
\end{prop}
\begin{proof} Let $P$ be a path-component of $\Map(X, Y)$. We claim $P$ is open and closed. Let $f$ be in $P$. We can assume $Y$ is connected. Thus, we have a sequence of compact subspaces $Y_1 \subset Y_2 \subset \cdots$ with union $Y$. Inductively, choose families $\alpha_i$ of open subsets of $Y$ such that $\alpha_i$ is an open cover of $Y_i$ and $\alpha_{i+1}|_{Y_i}$ is a refinement of $\alpha_i|_{Y_i}$ given by Lemma \ref{approximate homotopy}. Let $E_f$ be an $\alpha$-neighborhood of $f$ (the set of all maps $\alpha$-near to $f$). It is clearly open in the Whitney topology. We claim $E_f$ lies in $P$. Let $g$ be in $E_f$. If $h_i : X_i = f^{-1}(Y_i) \to Y_i$ is an $\alpha_i$-homotopy, then, since $f, g$ are $\alpha$-near, $f|_{X_i}, g|_{X_i}$ are $\alpha_i$-near. Thus, $h_i$ extends to an $\alpha_{i+1}$-homotopy $f|_{X_{i+1}}, g|_{X_{i+1}}$. Then define $h_t : X \to Y$ by $h_t|_{X_i} = h_{i + 1, t}$. Then $h$ is continuous since $X \times I$ is compactly generated; thus, the proof of the claim is complete.

To the see the last assertion, let $P' = P \cap \Map_{PA}(X, Y).$ We claim $P'$ lies in a path-component in $\Map_{PA}(X, Y)$. The first part of the proof goes through with PA-maps; thus, the connected components of $\Map_{PA}(X, Y)$ are the same as the path-components. Hence, it is enough to show $P'$ is connected; i.e., each continuous map $d : P' \to \{ 0, 1 \}$ is constant.

Suppose $d$ is not constant. Then by the definition of $P'$, we can find a continuous path $p$ lying in $P$ from some point in $d^{-1}(0)$ to some point $d^{-1}(1)$. For each $t$, $p(t)$ can be approximated by a PA-map; i.e., $p(t)$ is close to $P'$. Thus, we can find $f, g$ in $d^{-1}(0), d^{-1}(1)$ that can be taken to be arbitrary close. Then by the first part, there is a path $q : I \to P'$ with $q(0), q(1) = f, g$. Then $d \circ q$ is constant, a contradiction.
\end{proof}


Finally, we get:

\begin{corollary}\label{col:PA-homotopy} Let $f, g$ be PA-maps with the same domain and codomain. If there is a homotopy $h_t$ from $f$ to $g$ such that $t \mapsto h_t$ is continuous with respect to the Whitney topology, then there is a PA-homotopy from $f$ to $g$.
\end{corollary}

\subsection{Differential forms}
We now turn to some differential topology. We first need:

\begin{lemma}\label{generic smoothness} Let $X \subset \mathbb{R}^n$ be a subset and $f$ a function in $\mathcal{O}(X)$; i.e., it is continuous and definable on each compact definable subset.

Then there exists an open dense definable subset $U$ of $X$ such that $f$ is smooth on $U$; in fact, real-analytic on $U$.
\end{lemma}


\begin{proof} (This lemma is likely known but we could not find a source stating exactly the assertion instead of a general discussion of a construction of a stratification, from which the lemma should follow.)

Without loss of generality, we can assume $X = [0, 1]^n \subset \mathbb{R}^n$ is a unit cube. First assume $n = 1$. Then the assertion is known; see \cite[\S 2. Remark after the smooth monotonicity theorem]{van1999minimal}



For the general case, let $Y$ be the set of all points, at each of which $f$ is real-analytic in each variable. Then $Y$ is dense and definable by the first part; indeed, Hartogs's theorem implies that $f$ is real-analytic on a neighborhood of $y$ if it is so in each variable. The same argument also implies that $Y$ is open.
\end{proof}

\begin{remark}[regular locus] If $X$ is a compact PA-space contained in $\mathbb{R}^n$, then let $Z$ be the Zariski closure of $X$ in $\mathbb{R}^n$. Then we have the regular locus $Z_{reg}$ on $Z$ in the usual way (i.e., the set of points at each of which the local ring is regular).

Explicitly, $Z = V(I(X))$ where $I(X) \subset \mathbb{R}[x_1, \dots, x_n]$ is the ideal of all polynomials vanishing on $X$. Then $I(X)$ is generated by some finitely many $f_1, \dots, f_r$ and $Z = f^{-1}(0)$ for $f = (f_1, \dots, f_r) : \mathbb{R}^n \to \mathbb{R}^r$. By the Jacobian criterion, $Z_{reg}$ is exactly the locus where $df$ is surjective (the rank of it is $r$). So, by the inverse function theorem, $Z_{reg}$ is a manifold-without-boundary.

Then let
$$X_{reg} = Z_{reg} \cap \operatorname{int}_Z(X)$$
where $\operatorname{int}_Z(X)$ is the interior of $X$ in $Z$, which is nonempty since $\dim X= \dim Z$. Then $X_{reg}$ is a submanifold of $Z_{reg}$ without boundary; in particular, a manifold-without-boundary. Also, it is open and dense in $X$ since $Z_{reg}$ is dense in $Z$. Cf. Example \ref{cotangent space example}.

If $X$ is a PA-space in general with a defining cover $\iota_i : X_i \hookrightarrow \mathbb{R}^{n_i}$, then we let $X_{reg}$ be the union of all open definable subsets $U$ such that $U \subset \bigcup_i X_{i, \, reg}$ and $\iota_i \circ \iota_j^{-1}$ is smooth on $X_{i, \, reg} \cap X_{j, \, reg} \cap U$ for each pair $i, j$. Then, by Lemma \ref{generic smoothness}, $X_{reg}$ is open, dense and definable. It is called the \emph{regular locus} on $X$.

We note that the definition of the regular locus $X_{reg}$ depends on the embeddings $\iota_i$'s; in other words, two isomorphic PA-spaces might have different regular loci.

In practice, we shall usually simply fix some regular locus; Lemma \ref{O embed smooth} below implies that a choice of a regular locus does not matter much.
\end{remark}

With the above remark, we can now prove a more general version of Lemma \ref{generic smoothness}:

\begin{lemma}[generic smoothness]\label{O embed smooth} For each PA-space $X$,
$$\mathcal{O}(X) \hookrightarrow \varinjlim_{U \subset X_{reg}} \mathcal{C}^{\infty}(U), \, f \mapsto f|_U$$
where $U$ is an open dense definable subset.

Also, the injective limit on the right is independent of a choice of the regular locus $X_{reg}$.
\end{lemma}
\begin{proof} Since $X_{reg}$ is dense, this is immediate after Lemma \ref{generic smoothness}. Also, the limit is independent because of denseness.
\end{proof}

The lemma motivates us to define the sheaf $\Omega^p_{reg}$ on $X$ by
$$\Omega^p_{reg}(U) = \varinjlim_{V \subset X_{reg}} \Omega^p_{\mathcal{C}^{\infty}}(U \cap V)$$
for each open subset $U \subset X$. That is, $\Omega^p_{reg}$ is the sheaf of generically-smooth $p$-forms on $X$.

Then $\mathcal{O} \subset \Omega^0_{reg}$ by Lemma \ref{O embed smooth} and then we have
$$d : \mathcal{O} \to \Omega^1_{reg}$$
as the restriction of the usual exterior derivative $d : \Omega^0_{reg} \to \Omega^1_{reg}$. Let
$$\Omega^1_x = \mathcal{O}_x \im(d : \mathcal{O}_x \to \Omega^1_{reg, \, x}),$$
that is, the $\mathcal{O}_x$-module generated by $d(\mathcal{O}_x)$, and $\Omega^p_x = \wedge^p \Omega^1_x$. Then define $\Omega^p_X$ to be the sheaf whose stalks are exactly $\Omega^p_x$. Explicitly, a section $\omega$ in $\Omega^p_X(U)$ is such that it is locally a $\mathbb{Z}$-linear combination of sections of the form
$$x \mapsto f_{0, x} \, d(f_{1, x}) \wedge \cdots \wedge d(f_{p, x})$$
for $f_0, \dots, f_p$ local sections of $\mathcal{O}_X$ around $x$, where the subscript $-_x$ means taking the germ at $x$.

Finally, we define
$$d : \Omega^p_X \to \Omega_X^{p+1}$$
by the requirements
\begin{enumerate}[label=(\arabic*)]
\item For $p = 0$, this $d$ is the previous $d$.
\item $d(f_0 \, df_1 \wedge \cdots \wedge df_p) = df_0 \wedge df_1 \wedge \cdots \wedge df_p.$
\end{enumerate}

A key fact is that it commutes with pull-backs:

\begin{prop} Let $f : X \to Y$ be a PA-space morphism. Then we have the pull-back $f^* : \Omega_Y \to f_* \Omega_X$ that agrees with $f^* : \mathcal{O}_Y \to f_* \mathcal{O}_X$ and that commutes with $d : \Omega_X \to \Omega_X$.
\end{prop}
\begin{proof} By Lemma \ref{O embed smooth}, we have that $f : U \to Y_{reg}$ is smooth on some open dense subset $U \subset X_{reg} \cap f^{-1}(Y_{reg})$. Thus, we get
$$f^* : \Omega_{Y, \, reg} \to f_* \Omega_{X, \, reg}.$$
We shall check this restricts to $f^* : \Omega_Y \to f_* \Omega_X$. Such checking can be done on the stalk-level; that is, we have to check
$$f^* : \Omega^p_{Y, y} \to \Omega^p_{X, x}$$
is well-defined and commutes with $d$. First, note $f^*$ commutes with $d$ since it does so for $d$ on $\Omega_{X, reg}$ and $\Omega_{Y, reg}$. This in turns implies that the above map is well-defined.
\end{proof}


\subsection{Phase spaces}

Now, we want to define the phase space to a PA-space (i.e., the cotangent bundle in the smooth case). That notion is used in the discussion of differential operators or Fourier transformations. A usual way is to use a cotangent space but we shall now explain why that approach does not work.

Given a ring homomorphism $A \to R$, we recall that the module of K\"ahler differentials
$$(\Omega_{R/A}, d : R \to \Omega_{R/A})$$ is characterized by the universal property: each $A$-derivation $R \to M$ to a $R$-module $M$ factors through an $R$-linear map $\Omega_{R/A} \to M$.

In particular, with $R = \mathcal{O}_x$ and $A = k(x)$, the earlier $d : \mathcal{O}_x \to \Omega^1_x$ factors as
$$\mathcal{O}_x \to \Omega_{\mathcal{O}_x/k(x)} \overset{\eta}\to \Omega^1_x.$$
Here, we note that $\eta$ is surjective; indeed, the image of $\eta$ contains the generators but since the image is an $\mathcal{O}_x$-module, the image coincides with the whole module.

We have: with $k(x) = \mathcal{O}_x/\mathfrak{m}_x$,
$$\mathfrak{m}_x/\mathfrak{m}_x^2 \to \Omega_{\mathcal{O}_x/\mathbb{C}} \otimes k(x) \to \Omega_{k(x)/\mathbb{C}} \to 0$$
where the third term is zero since $k(x) = \mathbb{C}$. Thus, composing two surjections, we have the surjection
$$\mathfrak{m}_x/\mathfrak{m}_x^2 \to \Omega^1_x \otimes k(x).$$

On the other hand,

\begin{lemma} For each $x$ in $X$, we have: $\mathfrak{m}_{x}/\mathfrak{m}_{x}^2 = 0$.
\end{lemma}
\begin{proof} Assume $X$ is contained in $\mathbb{R}^n$, is compact and passes through the origin as well as $x = 0$. Given a class $[f]$ in $\mathfrak{m}_{0}/\mathfrak{m}_{0}^2$ represented by some real-valued $f$ in $\mathcal{O}_{X}(U)$, $U$ a neighborhood of $0$, we write
$$f = f_+ - f_-$$
where $f_+(x) = (f + |f|)/2$ and $f_{-}(x) = (f - |f|)/2$. Then $f_+, f_-$ are both in $\mathcal{O}_{X}(U)$ and in fact are in $\mathfrak{m}_{0}$. Thus, without loss of generality, we can assume $f \ge 0$. Then we have
$$f = \sqrt{f}\sqrt{f}.$$
Here, $\sqrt{f}$ is in $\mathcal{O}_0$; then in $\mathfrak{m}_0$. Thus, $f$ is in $\mathfrak{m}_0^2$; i.e., $[f] = 0$.
\end{proof}

\begin{remark}\label{cotangent space zero} In particular, $$\Omega^1_x \otimes k(x) = 0$$ for every $x$ in $X$. Note this \emph{does not mean} $\Omega^1_x = 0$ since Nakayama's lemma applies to a finitely generated module.
\end{remark}

Because of the above results, we shall construct the phase space to a PA-space somehow differently than simply using a Zariski cotangent space $\mathfrak{m}_x/\mathfrak{m}_x^2 = 0$. To guess a workable definition, first we consider the following simple example:

\begin{example}\label{cotangent space example} Let $X = V(xy) \cap [-1, 1]^2 \subset \mathbb{R}^2$ where $V(xy)$ is the zero locus of $xy$. Then the cotangent space $V_p$ to $X$ at a point $p$ should be the real vector space spanned by $(dx|_X)_p, (dy|_X)_p$, where the subscript $-_p$ means taking a germ not the restriction. (The restriction of any higher form to a point is zero for dimension reason.)

Thus, at a point $p \ne 0$, $\dim(V_p) = 1$ while at $p = 0$, $\dim(V_0) = 2$. Note if $p$ is in $X_{reg} = V(xy) \cap (-1, 1)^2 - 0$, then $V_p$ is the same as the cotangent space at $p$ to the manifold $X_{reg}$.

The Zariski closure $Z$ of $X$ is $V(xy)$ since clearly $I(X)$ is generated by $xy$. Thus, the above $V_p$ is the same thing as $\Omega_{Z/\mathbb{R}, \, p} \otimes k(p)$ where $\Omega_{Z/\mathbb{R}, \, p}$ is the cotangent sheaf to the variety $Z$ over $\mathbb{R}$ in algebraic geometry.

Let
$$\widetilde{X} = \{ (p, \xi) \mid p \in X, \, \xi \in V_p \}.$$

In the usual way, we identify $V_p$ with the tangent space at $p$ to $Z$. Then $p \times V_p \subset X \times \mathbb{R}^m \simeq X \times (\mathbb{R}^m)^*$ and so
$$\widetilde{X} \subset X \times (\mathbb{R}^m)^*.$$
\end{example}

We now simply generalize the above example:

\begin{construction}[phase space]\label{phase space construction} Given a compact PA-space $X \subset \mathbb{R}^m$, let $Z$ be the Zariski closure of $X$ in $\mathbb{R}^m$ and then for each point $p$ in $Z$, let $V_p = \Omega_{Z/\mathbb{R}, \, p} \otimes k(p)$ where $\Omega_{Z/\mathbb{R}}$ is the cotangent sheaf to $Z$ over $\mathbb{R}$. Explicitly, it is the $\mathbb{R}$-span of $(dx_1)|_{Z, \, p}, \dots, (dx_m)|_{Z, \, p}$ where $x_j$ denote the standard coordinates.

Then let
$$\widetilde{Z} = \operatorname{Spec}_Z(\operatorname{Sym}_{\mathcal{O}_Z} \Omega_{Z/\mathbb{R}}^{\vee})$$
where $-^{\vee}$ means the dual sheaf.

Note each fiber of $\widetilde{Z} \to Z$ is $\Omega_{Z/\mathbb{R}} \otimes k(p) = V_p$. Finally, we let $$\widetilde{X} = \widetilde{Z}|_X,$$
which we call the \emph{phase space} to $X$.

Note that $\Omega_{Z/\mathbb{R}}$ is globally finitely generated by $dx_1|_Z, \dots, dx_m|_Z$; i.e., we have the surjection $\mathcal{O}_Z[t_1, \dots, t_m] \to \operatorname{Sym}_{\mathcal{O}_Z}\Omega_{Z/\mathbb{R}}$, which gives
$$\operatorname{Spec}_Z(\operatorname{Sym}_{\mathcal{O}_Z} \Omega_{Z/\mathbb{R}}) \subset Z \times \mathbb{R}^n.$$

Suppose we are given a metric on $X_{reg}$ so that we can identify the cotangent spaces with the tangent spaces. Then the above gives:
$$\widetilde{X_{reg}} \hookrightarrow X_{reg} \times (\mathbb{R}^m)^*.$$
\end{construction}

\section{The sheaf of PA-forms}\label{sec:PA-forms}

The notion of PA-forms was introduced by \cite{Kontsevich}. Some aspects of the theory are detailed in \cite{Hardt}; in particular, a proof of the Poincar\'e lemma in the compact case. Here, we also give a version of the theory, which is somewhat more general.

\

Recall from Calculus \cite[Ch. 4.]{Spivak}: if $\omega$ is a differential $p$-form on a manifold $M$, $J = [0,1]^p$ a $p$-cube and $c : J \to M$ a continuous map, then we define
$$\int_c \omega := \int_J c^* \omega.$$
Note $c$ need not be smooth (since we can still write $c^* \omega = f \, dx_1 \wedge \cdots \wedge dx_p$ for some continuous function $f$ and then integrate this $f$. See also \cite[§ 4.8.]{Warner}.)

The above then extends by linearity to any $\mathbb{Z}$-linear combination $c = \sum n_i c_i$; i.e., $\int_c \omega = \sum n_i \int_{c_i} \omega$. This is an integration over a chain. A contour integral is a basic example.

A key fact in the above formalism is Stokes' formula, which fails without some degree of smoothness (continuity alone is not enough). Thus, first we introduce some smooth concepts. A \emph{smooth PA-manifold-with-boundary} or just a \emph{PA-manifold of dimension $n$} is a PA-space with the property that it has an open cover $\mathcal{U}$ together with open immersions that are PA-maps
$$\varphi_U : U \hookrightarrow \mathbb{R}^{n-1} \times \mathbb{R}_{\ge 0}$$
for $U \in \mathcal{U}$ such that each $\varphi_U|_{U \cap V} \circ \varphi_V^{-1}|_{\varphi_V(U \cap V)}$ is a smooth map. The \emph{orientation} on it and \emph{boundary} of it are defined just as in differential topology (e.g., \cite[Ch 5.]{Spivak}). If a PA-manifold $M$ is oriented, then the boundary of $M$ has induced orientation; namely, the one we get by requiring Stokes' formula hold \cite[Ch. 1., \S 3.]{DifferentialForms}.

\begin{remark}[product of PA-manifolds]\label{product of PA-manifolds} If $M, N$ are PA-manifolds, then, \emph{a priori}, there is an issue on a smooth structure on $M \times N$; i.e., $M \times N$ is only a manifold-with-corner. To mitigate this issue, we shall use \emph{straightening the corner} \cite[Proposition 2.6.2.]{Wall}, which says that there is a homeomorphism from $M \times N$ to a smooth manifold that is a diffeomorphism away from the corner and the latter manifold is unique up to diffeomorphisms. Examining the proof, we see that the PA-structure is preserved (the proof uses a collar and sometimes a collar is constructed by solving differential equations. But there is a more concrete construction of a collar; e.g., \cite[Ch. 4., § 6.]{Hirsch}). Therefore, through this procedure, we shall view $M \times N$ as a PA-manifold.
\end{remark}

Let $M$ be a compact oriented PA-manifold $M$ of dimension $p$. By the usual differential topology, we have the integration of continuous $p$-forms:
$$\int_M : \Omega^p_{\mathcal{C}^0(M)} \to \mathbb{C}.$$
Note, here, the smoothness is not needed and the continuity is enough; see \cite[§ 4.8.]{Warner}.

By a \emph{PA-chain of dimension $p$} over a commutative ring $\Lambda$, we mean a formal $\Lambda$-linear combination of PA-maps $c : M \to X$ where $M$ is an oriented compact PA-manifold of dimension $p$. Since the pull-back $c^* : \Omega(X) \to \Omega(M), \, c^*(\omega) = \omega \circ c$ is well-defined, the pairing $$(c, \omega) \mapsto \int_c \omega$$
is defined as before.

We say a PA-chain $c = \sum a_i c_i$ of dimension $p$ is \emph{degenerate} on an open subset $U \subset X$ if $\sum a_i \int_{c_i^{-1}(U)_{reg}} c_i^*\omega = 0$ for each $\omega$ in $\Omega^p(U)$. The \emph{support} of $c$ denoted by $\supp(c)$ is then defined to the complement of the union of open subsets $U$ on which $c$ is degenerate.

Finally, we let $C_p(X; \Lambda)$ be the group of \emph{nondegenerate} PA-chains of dimension $p$; i.e., it is the quotient of the group of all PA-chains modulo all degenerate PA-chains. The \emph{boundary operator} $\partial$ on $C(X; \Lambda)$ is defined by, for $c : M \to X$, taking $\partial c$ to be the restriction of $c$ to the boundary of $M$. Since the boundary of the boundary of a topological manifold is empty, we have $\partial^2 = 0$ and $(C(X; \Lambda), \partial)$ is a chain complex. Write $C(X)$ for $C(X; \mathbb{C})$.

\begin{lemma}[cf. \cite{Kontsevich} \S 8.2. Remark 9. (a)]\label{chain decomposes} If $Y_1,\,Y_2 \subset X$ are closed subsets with union $X$, then we have the Mayer-Vietoris sequence
$$0 \to C(Y_1 \cap Y_2) \to \oplus_{i=1}^2 C(Y_i) \to C(X) \to 0.$$
Moreover, it is split exact with sections given by a natural linear map
$$C_p(X) \to C_p(Y), \, c \mapsto c_Y$$
for each closed subset $Y$.
\end{lemma}
\begin{proof} Let $c : M \to X$ be in $C(X)$. Replacing $X$ by the image $c(M)$, we assume $c$ is surjective. Note $X$ is then compact and $c$ is a proper map.


Choose a finite stratification $\Sigma$ on $X$ such that $c$ is stratum-wise trivial and $\Sigma$ restricts to a stratification on $Y$. We can assume $Y$ has a triangulation $\Sigma_Y$ and each stratum on $Y$ is the union of simplexes in $\Sigma_Y$. Then, for each $\sigma \in \Sigma_Y$, by triviality of $c$ when restricted to strata, we have $c^{-1}(|\sigma|) \simeq |\sigma| \times c^{-1}(x_\sigma)$ (for an arbitrary choice of $x_\sigma \in \sigma$) and so, replacing $\sigma$ by $-\sigma$ if needed, the integration over $\sigma \times c^{-1}(x_\sigma)$ is the same as the integration over $c^{-1}(|\sigma|)$. We thus let $c_Y = \sum_{\sigma \in \Sigma_Y} \sigma \times c^{-1}(x_\sigma).$ It is clear that $c_Y$ then has the required property. Finally, the Mayer-Vietoris property implies the definition of $c_Y$ is independent of the choices involved in the construction.
\end{proof}

\begin{remark} In the lemma, we only consider the restriction to a \emph{closed} subspace since the restriction of a chain to an open subset would not have compact support in general. Probably for this reason, \cite[\S 8.2.]{Kontsevich} also defines the sheaf $C^{closed}$ of locally finite PA-chains with closed support (also known as the Borel--Moore sheaf), where ``closed support" means not necessarily compact support but only closed. For the definition of a locally finite chain, see \cite[Ch. I, \S 2.1.]{Borel}. In this paper, we do not need $C^{closed}$.
\end{remark}

The fundamental fact on PA-chains is the following:

\begin{lemma}[Stokes' formula]\label{Stokes} For $c$ in $C_p(X)$ and $\omega$ in $\Omega^{p-1}(U)$ for some open subset $U$ of $X$ containing $\supp(c)$, we have:
$$\int_c d\omega = \int_{\partial c} \omega.$$
\end{lemma}
\begin{proof} By linearity, we only need to check the formula on a compact oriented PA-manifold $M$ for $c : M \to X$. Thus, we shall assume $X = M$ and $c$ = the identity and show that $\int_M d\omega = \int_{\partial M} \omega$.

First we shall prove the special case when $M = \mathbb{R}^p \times [0, \infty)$. Let $\varphi$ be in $\mathcal{C}_c^{\infty}(B(0, 1))$ such that $\int \varphi \, dx = 1$ and then let $\varphi_{\epsilon}(x) = \epsilon^{-p} \varphi(x/\epsilon), \epsilon > 0$. For any continuous forms $\eta$ on $\mathbb{R}^p$, we can define the convolution $\eta * \varphi_{\epsilon}$ coefficient-wise. Then we have: for all $x$ in $M$,
$$\eta(x) = \lim_{\epsilon \to 0} (\eta * \varphi_{\epsilon})(x)$$
since, if $f$ is a coefficient in $\eta$,
$$|(f * \varphi_{\epsilon})(x) - f(x)| \le \sup_{|y - x| < \epsilon} |f(y) - f(x)| \sup |\varphi| \operatorname{vol}(B(x, 1)) \overset{\epsilon \to 0} \to 0$$
by continuity. (If $f$ is not continuous, we can use Lebesgue's differentiation theorem.)

Let $M_{\delta} = \mathbb{R}^{p - 1} \times [\delta, \infty), \, \delta > 0$. We shall show 
$$\int_{M_{\delta}} d\omega = \int_{\partial M_{\delta}} \omega.$$
Modifying $\omega$ outside $M_{\delta}$, without loss of generality, we shall assume $\omega$ is defined on the whole $\mathbb{R}^p$.
Then, by the usual Stokes' formula (or even just by a direct computation), we have:
$$\int_{M_{\delta}} d \omega * \varphi_{\epsilon} = \int_M d(\omega * \varphi_{\epsilon}) = \int_{\partial M_{\delta}} \omega * \varphi_{\epsilon}.$$
By Lebesgue's dominated convergence theorem, letting $\epsilon \to 0$ gives $\int_{M_{\delta}} d\omega = \int_{\partial M_{\delta}} \omega$ and then letting $\delta \to 0$ gives $\int_M d\omega = \int_{\partial M} \omega$. This completes the proof of the special case.

Finally, the general case follows by means of a partition of unity.
\end{proof}

\begin{warning} Without the continuity on $\omega$, Stokes' formula may fail, since for instance, $\omega$ may be modified to be zero on the boundary.
\end{warning}

We have the $L^1$-norm $\| \cdot \|_1$ on $\Omega^p(M)$ as follows. If $\omega$ is in $\Omega^p(M)$, then we write $\omega = f \, \mu$ where $\mu$ is the fixed volume form on $M$ and $f$ a continuous function on $M$. In local coordinates $x_j$, $\mu$ is a positive smooth function times $d x_1 \wedge \cdots \wedge dx_p$. Thus, if we let $\| \omega \|_1 = \int_M |f| \, \mu$, then it is non-negative and finite; i.e., is a norm. It is independent of a choice of a volume form (only depends on an orientation or an equivalence class of volume forms).

We then let $\Omega^p_{L^1}(M)$ be the completion of $\Omega^p(M)$ with respect to this norm. Since $\int_M : \Omega^p(M) \to \mathbb{C}$ is continuous; thus, by linearity, uniformly continuous with respect to $\| \cdot \|_1$, we have the unique continuous linear extension
$$\int_M : \Omega^p_{L^1}(M) \to \mathbb{C}.$$

We shall now define an analog of the space of locally integrable functions.

\begin{definition}\label{chain-wise integrable} By a \emph{set-theoretic $p$-form}, we mean a map $\omega : X \to E(\Omega_X^p)$ such that $X \overset{\omega}\to E(\Omega_X^p) \to X$ is the identity, where $E(\Omega_X^p)$ is the total space (espace \'etal\'e) of $\Omega_X^p$.

We say a set-theoretic $p$-form $\omega$ is \emph{chain-wise integrable} if $c^* \omega$ is integrable for each $c$ in $C_p(X)$. If $\omega$ is a chain-wise integrable $p$-form, then it determines a linear functional on $C_p(X)':= \Hom(C_p(X), \mathbb{C})$ by
$$\langle \omega, c \rangle = \int_c \omega.$$
We then say $\omega, \omega'$ are \emph{weakly equivalent} if $\langle \omega, c \rangle = \langle \omega', c \rangle$ for all $c$ in $C_p(X)$.

Then we let $\Omega^p_{chain-L^1}(X)$ be the space of all weak-equivalence classes of chain-wise integrable $p$-forms.
\end{definition}

For example, every continuous $p$-form is chain-wise integrable and can be viewed as an element of $\Omega^p_{chain-L^1}(X)$, since, by Lebesgue's differentiation theorem, two weakly-equivalent continuous forms are the same.

We have:
$$\Omega^p_{chain-L^1}(X) \hookrightarrow C_p(X)'$$
given by
$$\omega \mapsto \left(c \mapsto \int_c \omega \right).$$ The differential on $C_p(X)'$ is defined as the transpose of $\partial$; i.e., $\langle d \omega, c \rangle = \langle \omega, \partial c \rangle$. By Stokes' formula (Lemma \ref{Stokes}), this differential agrees with the differential $d$ on $\Omega(X)$.

For a PA-space $X$, we give $C_r(X)$ the initial topology; i.e., the coarsest topology such that
\begin{enumerate}[label=(\alph*)]
    \item $c \mapsto \int_c \omega$ is continuous for every $\omega$ in $\Omega^r(X)$,
    \item if $c_i$ denotes a defining basis element on $C(X)$, then $c_i^* : C(X) \to \mathbb{Z}$ is continuous.
\end{enumerate}

We now introduce a family version of a chain:

\begin{definition}\label{r-family} By a \emph{continuous $r$-family of chains in fibers over $X$} or just an \emph{$r$-family over $X$}, we mean a continuous map
$$\gamma : X \to C_r(Y)$$
for some proper PA-map $\pi : Y \to X$, which by proper we mean the pre-images of compact sets are compact, such that 
\begin{enumerate}[label=(\arabic*)]
    \item $\supp(\gamma(x)) \subset \pi^{-1}(x)$, and
    \item it is in the image of
    \begin{align*}
    \bigoplus_M \mathbb{Z}\Map_{\mathrm{PA}}(X \times M, Y) &\hookrightarrow \Map(X, C_r(Y)) \\
    \sum n_i \widetilde{\gamma_i} &\mapsto (x \mapsto \sum n_i \widetilde{\gamma_i}(x, \cdot)).
    \end{align*}
\end{enumerate}

The \emph{boundary} $\partial \gamma$ of $\gamma$ is the $(r - 1)$-family defined as
$$X \overset{\gamma}\to C_r(Y) \overset{\partial}\to C_{r-1}(Y).$$
Note $\partial$ here is continuous since $\langle \omega, \partial c \rangle = \langle d\omega, c \rangle$ and the latter is continuous in $c$.

Also, if $f : X' \to X$ is a PA-map, and $Y'= Y \times_f X'$, then we define the pullback $f^* \gamma$ by using $\Map(X \times M, Y) \overset{f^*}\to \Map(X' \times M, Y')$.
\end{definition}

The condition (2) is needed to show a PA-form is generically smooth later (Proposition \ref{PA approximation}).

The integration over a chain then generalizes to the integration over chains in fibers (which also generalizes integration along fibers), in the following sense:

\begin{lemma}[integration along chains in fibers]\label{integration along fibers} For an $r$-family $\gamma$ with proper $\pi : Y \to X$, there is a unique linear operator
$$\int_{\gamma} : \Omega^{p}_{chain-L^1}(Y) \to \Omega^{p - r}_{chain-L^1}(X)$$
(by convention, $\Omega_{chain-L^1}^q = 0$ for negative $q$)
such that
\begin{enumerate}[label=(\arabic*)]
\item For each $\eta$ in $\Omega^{r}_{chain-L^1}(Y)$,
$\displaystyle \int_{\gamma} \eta$ is $$\displaystyle x \mapsto \int_{\gamma(x)} \eta|_{\pi^{-1}(x)}.$$
\item For each subspace $A \subset X$,
$\left. \left( \int_{\gamma} \omega \right) \right \vert_A = \int_{\gamma|_A} \omega|_{\pi^{-1}(A)}.$
\item (projection formula) for a continuous form $\eta$ on $Y$ and a continuous form $\sigma$ on $X$,
$$\int \eta \wedge \pi^* \sigma = \left(\int_{\gamma} \eta \right) \wedge \sigma.$$
In other words, $\int_{\gamma}$ is linear with respect to the right $\Omega_{\mathcal{C}^0}(X)$-module structures.
\end{enumerate}

Moreover, the above operator has the following properties: if $f : Y' \to Y$ is proper, then $$\int_{f_* \gamma} \omega = \int_{\gamma} f^*\omega.$$
Similarly, if $f : X' \to X$ is a PA-map, then
$$f^* \int_{\gamma} \omega = \int_{f^* \gamma} f^* \omega$$
where $f^* \omega$ is the pull-back along $Y' = Y \times_X X' \to Y$.
\end{lemma}

Note: some authors such as Bott--Tu use the left $\Omega(X)$-module structure on $\Omega(Y)$ instead of the right one here; so our definition differs from theirs by a sign.

\begin{proof}[Proof of Lemma \ref{integration along fibers}] Let $Z$ be the closure of $\bigcup_{x \in X} \supp(\gamma(x))$. Replacing $Y$ by $Z$, we shall assume $r$ is the dimension of each fiber $\pi^{-1}(x)$.

Since the problem is local, we can assume $X, Y$ are compact definable subsets of Euclidean spaces. Since $\pi$ factors as $Y \hookrightarrow \Gamma_{\pi} \overset{\mathrm{proj}} \to X$, we can assume $\pi$ is real-algebraic. Then, by \cite[Part 1. Ch 1. \S 1.7. Theorem]{StratifiedMorse}, we can find definable Whitney stratifications on $X, Y$ such that for each stratum $A$ on $X$, the pre-image $\pi^{-1}(A)$ is a union of strata $B$'s and $\pi : B \to A$ is a proper (smooth) submersion. By Ehresmann's lemma (\cite[Ch. 6. § 2. Theorem 2.2.]{Hirsch}), such $\pi : B \to A$ is smoothly locally trivial. We assume each $A$ is so small that $\pi : B \to A$ is trivial; i.e., $\pi$ factors as $B \simeq A \times F \overset{\mathrm{pr}}\to A$. Also, we note that since $\Omega_{\mathcal{C}^0}(M)$ is dense in $\Omega_{L^1}(M)$, it is enough to define the integration for continuous forms. Hence every form $\omega$ on $A$ decomposes into a $\mathbb{Z}$-linear combination of the forms like $\eta \wedge \pi^* \sigma$ where $\eta$ is of type $s \leq r$.

Having done the above reduction, we first prove the uniqueness. By (3), we have:
$$\left. \left( \int_{\gamma} \omega \right) \right \vert_A = \int_{\gamma|_A} \omega|_A.$$
Here, if $\omega = \eta \wedge \pi^* \sigma$ on $A$, then by the projection formula
$$\int_{\gamma|_A} \omega|_A = \left(\int_{\gamma|_A} \eta \right) \sigma$$
where $\int_{\gamma|_A} \eta$ is uniquely given by (2). This proves the uniqueness. The existence also holds since the same formula defines $\int_{\gamma}$.

We shall now show $\int_{\gamma} \omega$ is integrable chain-wise; i.e., $c^* \int_{\gamma} \omega$ is integrable. We can write
$$c^* \int_{\gamma} \omega = \sum_A c_A^* \int_{\gamma|_A} \omega$$
where $c_A : c^{-1}(A) \to A$. When we decompose $\omega$ on $A$, it has integrable coefficients. Hence, each $c_A^* \int_{\gamma|_A} \omega$ is integrable.
\end{proof}

\begin{example}[\cite{Kontsevich} \S 8.4. Example 2.] Let $F = \mathbb{R}$ and $\pi: Y = \mathbb{R} \times F \to X = \mathbb{R}, (x, y) \mapsto x$ be the projection. Define $\gamma : X \to C_0(Y)$ by 
$$\gamma(x) = \{ (x, x) \} + \{ (x, -x) \}$$
if $x > 0$ and $\gamma(x) = 2\{ (x, 0) \}$ if $x \le 0$. For a function $f$ in $\mathcal{O}(Y)$, we have:
$$\left( \int_{\gamma} f \right)(x) = \int_{\gamma(x)} f|_{x \times F} = f(x, x) + f(x, -x),$$
which goes to $2f(0, 0)$ as $x \to 0$ since $f$ is continuous.
\end{example}

We can now define:

\begin{definition}[PA-forms] Let $U$ be an open subset of $X$.

A \emph{PA-form on $U$ of type $p$} is a linear functional $l$ on $C_p(U)$ such that for each compact subset $K$ of $U$, there exist an $r$-family $\gamma: U \to C_r(Y)$ and $\omega$ in $\Omega_{p+r}(Y)$ so that
$$\langle l, c \rangle = \int_c \int_{\gamma} \omega$$
for each $c$ in $C_p(X)$ with $\supp(c) \subset K$.

We then let $\Omega_{X, PA}(U)$ be the set of all PA-forms on $U$ (which will turn out to be a vector space).
\end{definition}

Note we \emph{do not} require the linear functional $l$ above be given by a single integration. This flexibility is needed to ensure $\Omega_{PA}$ is a sheaf.

\begin{remark}\label{weak equivalence PA-form} We note that $\int_{\gamma} \omega$ is continuous on each stratum $A$ of $X$ when $X$ is given a stratification as in the proof of Lemma \ref{integration along fibers}; i.e., it is in $\Omega_{\mathcal{C}^0}(A)$. In particular, if $l$ is given by $\int_{\gamma} \omega$ with respect to some compact subset $K$ and $l \sim 0$ on $K$; i.e., $\langle l, c \rangle = 0$ for all $c$ in $C(X)$ with $\supp(c) \subset K$, then $\int_{\gamma} \omega = 0$ on $K$.
\end{remark}

The next lemma says in particular that $\Omega_{X, PA}(U)$ is a vector space.

\begin{lemma}\label{Omega_PA a module} The above $\Omega_{X, PA}(U)$ is naturally an $\mathcal{O}_X(U)$-module.
\end{lemma}
\begin{proof} We first note that if $l_1, l_2$ are PA-forms on $U$, then $l_1 + l_2$ is a PA-form on $U$. Indeed, that follows from
$$\int_{\gamma_1} \omega_1 + \int_{\gamma_2} \omega_2 = \int_{\gamma_1 \sqcup \gamma_2} (p_1^* \omega_1 + p_2^* \omega_2)$$
where $\gamma_1 \sqcup \gamma_2$ is $X \to C(Y_1 \sqcup Y_2), \, x \mapsto \gamma_1(x) \sqcup \gamma_2(x)$ and $p_i : Y_1 \sqcup Y_2 \to Y_i$ are the retracts of $Y_i \hookrightarrow Y_1 \sqcup Y_2$.

Hence, $\Omega^p_{PA}(U)$ is a vector space.

Next, if $\lambda$ is in $\mathcal{O}(U)$ and $l$ in $\Omega^p_{PA}(U)$ given by $\int_{\gamma} \omega$ with $\pi: Y \to U$ with respect to some fixed compact subset $K$ of $U$, then we define $\lambda \, l$ by
$$\langle \lambda \, l, c \rangle = \int_c \lambda \int_{\gamma} \omega = \int_c \int_{\gamma} (\pi^* \lambda) \omega.$$
To see this is well-defined; i.e., independent of a choice of $K$, suppose $l$ is given by $\int_{\gamma'} \omega'$ with respect to another compact subset $K' \subset U$. Then $\int_c \int_{\gamma} \omega = \int_c \int_{\gamma'} \omega'$ for all chains $c$ lying in $K \cap K'$ implies $\int_{\gamma} \omega$ and $\int_{\gamma'} \omega'$ agree on $K \cap K'$ by Remark \ref{weak equivalence PA-form}. It follows they both agree there after they are multiplied by $\lambda$.
\end{proof}

If $V \subset U$ are open subsets in a PA-space $X$, then the restriction
$$\Omega^p_{PA}(U) \to \Omega^p_{PA}(V)$$
is given by restricting a linear functional $l$ to the subspace $C_p(V) \subset C_p(U)$.

\begin{prop}\label{soft sheaf} $\Omega^p_{PA}$ is a soft sheaf; in particular, a sheaf.
\end{prop}
\begin{proof} The nontrivial claim here is that $\Omega^p_{PA}$ is a sheaf (``soft'' is clear by a cut-off argument). We shall verify the sheaf axioms: given a family of open subsets $U_i$ with union $U$, (1) if $l|_{U_i} = 0$, then $l|_U = 0$ and (2) if $l_i$ are in $\Omega^p_{PA}(U_i)$ such that $l_i|_{U_i \cap U_j} = l_j|_{U_i \cap U_j}$, then there is a section $l$ in $\Omega^p_{PA}(U)$ such that $l|_{U_i} = l_i$.
Without loss of generality, we assume the open covering $\{U_i\}_i$ is locally finite.

In proving (1), by linearity, it is enough to show $\langle l, c \rangle = 0$ for each basis element $c$ of $C_p(U)$; i.e., $c : M \to U$. Fix such a $c$. As $\im(c)$ is compact, pick a finite set $I'$ such that $U_i, i \in I$ cover $\im(c)$. Let $1 = \sum_{i \in I'} \lambda_i$ be a finite partition of unity with $\supp(\lambda_i) \subset U_i$. Let $V = \bigcup_{i \in I'} U_i$. Since the support of $c$ is contained in $V$, we have
$$\langle l, c \rangle = \langle l|_V, c \rangle = \sum_{i \in I'} \langle \lambda_i l|_V, c \rangle.$$

Now, by Lemma \ref{chain decomposes}, let $c_i$ be the restriction of $c$ to some definable compact neighborhood of $\supp(\lambda_i)$ in $U_i$. Then
$$\langle \lambda_i l, c \rangle = \int_c \lambda_i \int_{\gamma} \omega = \int_{c_i} \lambda_i \int_{\gamma} \omega = \langle \lambda_i l, c_i \rangle = 0$$
since $\im(c_i) \subset U_i$ and since $l|_{U_i} = 0$.

The proof of (2) is similar. Take a partition of unity $\lambda_i$ for the cover $U_i$. Note $\lambda_i l_i$ can be viewed as an element of $\Omega_{PA}(U)$ since, as above, we can shrink a chain $c$ on $U$ to one lying in $U_i$ without changing the integration. Then we let $l = \sum_i \lambda_i l_i$, which is a well-defined linear functional on $C_p(U)$ since $\sum_i \langle \lambda_i l_i, c \rangle$ is a finite sum for each $c$. We check $l$ has the required form with respect to a given compact subset $K \subset U$. Since $K$ is compact and the cover $U_i$ is locally finite, we can choose a neighborhood $V$ of $K$ in $U$ such that only finitely many $U_i$'s intersect $V$. Since $\Omega_{V, PA}(V)$ is an $\mathcal{O}_V(V)$-module (Lemma \ref{Omega_PA a module}), the finite sum $l|_V = \sum_i \lambda_i l_i|_V$ is in $\Omega_{V, PA}(V)$. Then, with respect to $K$, $l|_V$ is given by some $\int_{\gamma} \omega$. In particular, $l$ has the required form with respect to $K$.
\end{proof}

\begin{remark}\label{support condition} Here is why we cannot get a sheaf if a PA-form is always allowed to be given by a single integration.

Consider an increasing sequence of open subsets $U_1 \subset U_2 \subset \cdots$ and PA $p$-forms $l_i$ on $U_i$ given by $\langle l_i, c \rangle = \int_c \int_{\gamma_i} \omega_i$ such that $l_i|_{U_j} = l_j, \, j \le i$. Now, it could happen that each $\gamma_i$ is an $r_i$-family with $r_i \to \infty$. In that case, there is little or no hope to find a $\gamma$ that is obtained by gluing $\gamma_i$'s in some sense.

One alternative is to consider an $\infty$-family with infinite-dimensional spaces but here we prefer not to pursue that possibility.
\end{remark}

\begin{example} We have $\Omega_X^p(U) \hookrightarrow \Omega^p_{X, PA}(U)$ by
$$\omega \mapsto (c \mapsto \int_c \omega = \int_c \int_{\gamma} \omega)$$
where $\gamma : X \to C(X)$ is given by the identity.
\end{example}

The next log example can be read as a motivation for the definition of a PA-form.

\begin{example}[log]\label{log example} Let $\sigma = dx/x$ be a $1$-form on the open interval $X = (0, \infty)$. If there \emph{were} a $\psi$ on $U$ such that $d \psi = \sigma$, then $\psi = \log$. Now, this $\psi = \log$ is not in $\Omega(X)$ as the graph of the exponential function is not definable. In particular, this means that the Poincar\'e lemma fails for the complex of sheaves $(\Omega_X, d)$.

But $\log$ is in $\Omega_{PA}(X)$. Indeed, if $c: * \to X$ such that $c(*) = x$ in $K = [a, b]$, then we have $$\log(x) = \int_c \log = \int_c \int_1^x \frac{dt}{t} = \int_c \int_{\gamma} \frac{dt}{t}$$
where $\pi : Y = X \times K \to X$ a projection with the coordinate $t$ on $K$ and $\gamma: X \to C(Y), \, \gamma(x) = (x, [1, x])$.

Similarly (cf. \cite[Example 1]{Kontsevich}), we have that $x \log x$ is a PA-form on $X = [0, \infty)$ since
$$x \log x = \int_1^x \frac{x}{t} \, dt = \int_1^x y \, dt$$
with $\pi : \mathbb{R}_{\ge 0} \times [0, b] \to X, \, (y, t) \mapsto yt$.
\end{example}

\begin{example}[non-square-integrable]\label{non-square-integrable} While a PA-form has locally integrable coefficients, those coefficients need not be locally square-integrable. Indeed, consider $\omega = d \sqrt{x}$ on $X = [0, 1]$. Then $\omega = \frac{dx}{\sqrt{x}}$ where $1/\sqrt{x}$ is integrable but not square-integrable.
\end{example}

\begin{lemma}\label{C(U) dense} If $U \subset X$ is an open dense definable subset, then $C_p(U)$ is dense in $C_p(X)$.

(In the above, we cannot drop ``definable'' since, otherwise, the boundary of $U$ may be ugly.)
\end{lemma}
\begin{proof} This lemma is essentially the same as \cite[Propositions 4.3.]{Hardt} and we adapt their proof.

Let $c$ be a chain in $C(X)$. The support of $c$ has non-empty intersection with only finitely many elements of the defining (compact) cover of $X$; call them $X_1, \dots, X_n$. As $X_1 \cup \cdots \cup X_n$ is compact, we can pick a definable open neighborhood $X_1 \cup \cdots \cup X_n \subset V \subset X$ with the property that $\overline{V}$ intersects only finitely many of elements of defining cover, called $X_1,\,X_2,\cdots, \,X_n, \cdots , X_m$. Now $U \cap V \subset V$ is dense and $V$ can be viewed as a definable subset of some Euclidean spaces $\mathbb{R}^n$.




Hence, replacing $X$ by $\bigcup_{i = 1}^m X_i$, we can assume $X \subset \mathbb{R}^n$ is a compact definable subset and $c : A \to X$ where $A \subset \mathbb{R}^n$ is a compact definable subset. Now, by \cite[Th\'eor\`em 9.2.1.]{Bochnak}, $A$ admits a triangulation. So, for every $\omega$ in $\Omega^p(X)$, we have:
$$\int_c \omega = \sum_{\sigma} \int_{\sigma} \omega$$
where $\sigma$ runs over $\Delta^p \overset{\tau}\to A \overset{c}\to X$, $\tau$ a $p$-simplex in the triangulation of $A$. (As a matter of convention, we prefer closed simplicies to open simplicies.) Thus, replacing $c$ by $\sigma$, we can assume $c = \sigma$ is a $p$-simplex.

Now, by \cite[Th\'eor\`em 9.2.1.]{Bochnak}, we can find a triangulation of $X$ that refines the stratification $\{ U, X - U \}$ of $X$ and that contains $\sigma$, up to refining the stratification on $A$. If $\sigma$ lies in $U$, we are done. If $\sigma$ lies in $X - U$, then $\sigma$ is a face of some simplex $\sigma'$ whose relative interior lies in $X - U$. Then, we can find a definable homotopy $h_t : \Delta^p \to X$ such that $h_0 = \sigma$ and $h_t$ has image in $X - U$ for $t > 0$. Clearly, $h_t \to h_0$ as $t \to 0$ in the weak topology on $C_p(U)$.
\end{proof}

The next proposition often allows us to reduce to the smooth case.

\begin{prop}[\cite{Kontsevich} \S 8.4.]\label{PA approximation} We have the natural embedding:
$$\Omega_{PA}(X) \hookrightarrow \varinjlim_U \Omega_{\mathcal{C}^{\infty}}(U)$$
given by $\alpha \mapsto \alpha|_U$, where $U \subset X_{reg}$ are all the open dense definable subsets.

Also, $\Omega_{PA}(X) \subset $ the continuous dual of $C(X)$.
\end{prop}
\begin{proof} First, we note that a PA-form is weakly continuous; i.e., $\langle \int_{\gamma} \omega, \cdot \rangle$ is a continuous linear functional on $C(X)$. Indeed, from the proof of Lemma \ref{integration along fibers}, we see that $\int_{\gamma} \omega$ is bounded stratum-wise when the strata are relatively compact, since $\omega$ is bounded on each compact subset. Thus, for each $c$ in $C(X)$, we have that $c^* \int_{\gamma} \omega$ is bounded since $\supp(c)$ intersects only finitely many strata. Now, take $c$ to be $c : M \times I \to X$. Then $c^* \int_{\gamma} \omega$ is bounded and thus, by Lebesgue's dominated convergence theorem, $\int_{c_t} \int_{\gamma} \omega \to \int_{c_0} \int_{\gamma} \omega$ as $t \to 0$. The injectivity then follows from Lemma \ref{C(U) dense}.

Next, we shall show that each PA-form is smooth on some open dense definable subset of $X_{reg}$. Suppose we are given a PA-form of the form $\int_{\gamma} \omega$ relative to $\pi : Y \to X$. Then $\pi$ is locally trivial on some open dense definable subset of $X$. By the sheaf property, we can thus assume $\pi : Y \to X$ is trivial; i.e., $Y = X \times F$ and $\pi$ the projection.

By linearity, we assume $\gamma$ is given by $\widetilde{\gamma} : M \times X \to Y$; i.e., $\gamma(x)(m) = \widetilde{\gamma}(x, m)$. Then we write:
$$\int_{\gamma(x)} \omega_x = \int_M \gamma(x)^* \omega_x.$$
Then, by the differentiation under the integral sign, $\int_{\gamma(x)} \omega_x$ is smooth in $x$ in some open dense definable.
\end{proof}

The proposition says in particular that a PA-form is generically determined; i.e.,

\begin{corollary}\label{cor:PA-form generic} For each $\alpha$ in $\Omega_{PA}(X)$, if $\alpha = 0$ on an open dense subset, then $\alpha = 0$. Similarly, for each $\alpha$ in $\Omega_{PA}(X)$, if $X$ is given a (locally finite) stratification and $\alpha = 0$ on each stratum, then $\alpha = 0$.
\end{corollary}

The next proposition may be thought of as a common generalization of the fundamental theorem of calculus and the differentiation under the integral sign (in a special case, it also reduces to Stokes' formula):

\begin{prop}[cf. \cite{Hardt} Lemma 5.23.]\label{FTC} For each $\omega$ in $\Omega^{p+r-1}(Y)$ and $\gamma$ an $r$-family, we have:
$$d \int_{\gamma} \omega = (-1)^{r-1} \int_{\partial \gamma} \omega + (-1)^r\int_{\gamma} d\omega$$
where the formula holds in the weak sense; i.e., the integrals are linear functionals and $d$ is the transpose of the boundary operator $\partial$.
\end{prop}
\begin{proof} It is enough to prove the proposition on some open dense subset $U \subset X$. By Proposition \ref{PA approximation}, the proof then reduces to the case when $\omega$ is smooth.

As in the proof of Proposition \ref{PA approximation}, at the expense of changing $\omega$, we can assume $\gamma$ is constant; i.e., $\gamma(x)$ is independent of $x$. 

Then we can write $d = d_X + d_F$ where $d_X, d_F$ are the differentials on $X, F$. Then, by the differentiation under the integral sign (applicable by smoothness), we easily see
$$\int_{\gamma} d \omega = \int_{\gamma} (d_X \omega + d_F \omega) = (-1)^r d \int_{\gamma} \omega + \int_{\partial \gamma} \omega.$$
\end{proof}

Since $\Omega^p_{PA}(U)$ is a subspace of the $\mathbb{C}$-dual of $C_p(U)$, we can apply the differential $d$ to $\Omega^p_{PA}(U)$. Then

\begin{corollary}\label{differential on PA} The differential $d$ on $\Omega_{PA}$ is well-defined; that is, $d(\Omega_{PA}(U)) \subset \Omega_{PA}(U)$.
\end{corollary}
\begin{proof} Suppose we have a PA-form $l$ given as $\int_{\gamma} \omega$. Then
$$\langle dl, c \rangle = \int_c (-1)^{r-1} \int_{\partial \gamma} \omega + \int_c (-1)^r \int_{\gamma} d \omega.$$
\end{proof}

\begin{remark}[pullback] Let $f : X \to Y$ be a morphism of PA-spaces. Then $f^* : \Omega_{PA}(Y) \to \Omega_{PA}(X)$ is given by
$$\langle f^*\alpha, c \rangle = \langle \alpha, f_* c \rangle$$
where $f_*(c) = f \circ c$. It is well-defined since, if $\alpha = \int_{\gamma} \omega$,
$$\langle \alpha, f_* c \rangle = \int_c f^* \int_{\gamma} \omega = \int_c \int_{f^* \gamma} f^*\omega$$
by Lemma \ref{integration along fibers}. (This argument also shows that the pull-back defined here agrees with the usual pull-back of forms.)

Next, we shall show $f^*$ commutes with $d$, or dually, we shall show $f_*$ commutes with $\partial$. But the latter holds since, for $c : M \to X$,
$$\partial (f \circ c) = (f \circ c)|_{\partial M} = f \circ c|_{\partial M} = f \circ \partial c.$$
\end{remark}

Also, we have:

\begin{remark}[direct image]\label{direct image} Given a projection $\pi : X \times F \to X$, $F$ an oriented compact PA-manifold of pure dimension, we define the direct image $\pi_* : \Omega_{PA}(X \times F) \to \Omega_{PA}(X)$ by
$$\pi_* \left(\int_{\gamma} \omega \right) = \int_{\gamma \times F} \omega$$
where if $\gamma : X \times F \to C(Y)$ is given by $\widetilde{\gamma} : X \times F \times M \to Y$, then since $M \times F$ is an oriented compact PA-manifold (Remark \ref{product of PA-manifolds}), via $M \times F \simeq F \times M$, $\widetilde{\gamma}$ determines $\gamma \times F : X \to C(Y)$.

The notation $\gamma \times F$ might look jarring but it is consistent with a product of spaces.

For each $c : M \to X$ in $C_p(X)$, we have:
$$\partial(c \times F) \sim \partial c \times F + (-1)^p c \times \partial F$$
where $c \sim c'$ means that $\langle \cdot, c \rangle = \langle \cdot, c' \rangle$ on $\Omega(X)$. Indeed, as a set, $\partial(M \times F)$ is the union of $\partial M \times F$ and $M \times \partial F$. It is then tedious but easy to see that, away from their intersection, $\partial(M \times F)$ agrees locally either with $\partial M \times F$ or with $(-1)^p M \times \partial F$, where $-1$ means opposite orientation.

It follows:
$$\partial (\gamma \times F) \sim \partial \gamma \times F + (-1)^r \gamma \times \partial F$$
where $\sim$ holds pointwise. Note also that, by Fubini's theorem, we have:
$$\int_{\gamma \times F} \omega = \int_F \int_{\gamma} \omega.$$
Then, by the above identity and using Proposition \ref{FTC} twice, for $\alpha = \int_{\gamma} \omega$, we have: for $s = \dim F$,
\begin{align*}
d \pi_*(\alpha) &= (-1)^{r + s - 1} \int_{\partial \gamma \times F} \omega + (-1)^{s-1}\int_{\gamma \times \partial F} \omega + (-1)^{r + s} \int_{\gamma \times F} d \omega \\
&= (-1)^{s-1}\int_{\partial F} \alpha + (-1)^s \pi_*(d \alpha).
\end{align*}
\end{remark}

We can now state and prove the main result of this section:

\begin{prop}[Poincar\'e lemma]\label{poincare} $\mathbb{C}$ viewed as a constant sheaf on $X$, the complex of sheaves
$$0 \to \mathbb{C} \to \Omega_{PA}^0 \overset{d}\to \Omega_{PA}^1 \overset{d}\to \cdots \to 0$$
is exact; i.e., it is a resolution of $\mathbb{C}$.
\end{prop}
\begin{proof} Here, we shall repeat a standard proof.
 
By Remark \ref{direct image}, for any PA-form $\alpha$ on $X \times I$, we have:
$$d\pi_* \alpha + \pi_* d\alpha= \alpha_1 - \alpha_0.$$

By Proposition \ref{locally contractible} and by Corollary \ref{col:PA-homotopy}, shrinking $X$, we can assume $X$ is contractible; thus, we have a PA-homotopy $h_t : X \to X$ such that $h_1$ is the identity and $h_0$ is a constant map. Then, given a $\omega$ in $\Omega_{PA}^p(X)$, taking $\alpha = h^* \omega$ and $J = \pi_* h^*$, we have:
$$d J \omega + Jd \omega = \omega - h_0^* \omega.$$
If $p \ge 1$, then $h_0^* \omega = 0$ since there is no higher form at a point. Thus, if $\omega$ is closed, then the above says $\omega$ is exact.

If $p = 0$, then $J\omega$ is zero since it has type $-1$ and so the above formula says $\omega$ is constant when $\omega$ is closed.

\end{proof}

\begin{corollary}\label{PA forms cohomology} $\operatorname{H}_{sing}^i(X; \mathbb{C}) = \operatorname{H}^i(0 \to \Omega^0_{PA}(X) \overset{d}\to \Omega^1_{PA}(X) \overset{d}\to \cdots)$.
\end{corollary}
\begin{proof} As in Lemma \ref{singular cohomology}, we can use $(\Omega_{PA}, d)$ to compute the sheaf cohomology of $X$ with values in $\mathbb{C}$ and then by that lemma, the sheaf cohomology coincides with the singular cohomology.
\end{proof}

We record the following result, which we frequently use:

\begin{lemma}\label{the integration of a PA-form} If $\alpha$ is in $\Omega_{PA}(X)$, then $\int_{\gamma'} \alpha$ is in $\Omega_{PA}(X')$ for $\gamma' : X' \to C(X)$ with proper $X' \to X$.

In particular, if $M \subset X_{reg}$ is an oriented PA-submanifold, then the integration $\int_M \alpha$ is defined, though it possibly diverges.
\end{lemma}
\begin{proof} Suppose $\alpha = \int_{\gamma} \omega$. We shall define $\gamma \times \gamma'$ in the same way we defined $\gamma \times F$ in Remark \ref{direct image}. Namely, if $\gamma$ is given by $X \times M \to Y$ and $\gamma'$ by $X' \times M' \to X$, then let $\gamma' \times \gamma$ be given by $X' \times M' \times M \overset{\gamma' \times \operatorname{id}} \to X \times M \overset{\gamma}\to Y$.
Then we have a version of Fubini's theorem:
$$\int_{\gamma' \times \gamma} \omega = \int_{\gamma'} \int_{\gamma} \omega,$$
which implies the assertion.
\end{proof}

\begin{remark} Let $f : M \to X$ be a proper morphism from a not-necessarily-compact but oriented PA-manifold $M$. Then
$$\int_M f^* : \Omega_{PA, \, cpt}(X) \to \mathbb{C}$$
is well-defined. Indeed, we have $f^* : \Omega_{PA, \, cpt}(X) \to \Omega_{PA, \, cpt}(M)$.
\end{remark}

\section{Currents}\label{sec:currents}

In this section, we introduce the sheaf $\Omega'_{PA}$ of currents on a PA-space and prove a version of the Poincar\'e lemma for currents.

We recall (Proposition \ref{PA approximation}) that
$$\Omega_{PA}(X) \subset C(X)'$$
where the superscript $'$ means the continuous dual. In particular, $\Omega_{PA}(X)$ ``can" be given the weak-* topology as the subspace topology. But with respect to the weak-* topology, elements of $C(X)$ are the only continuous linear functionals on $C(X)'$ and at least for $p = 0$, $\Omega_{PA}^0(X)$ is dense in $C_0(X)'$. Thus, to get more continuous linear functionals, we shall use a finer topology, the sup norm topology instead. Namely,

\begin{remark}[sup norm]\label{sup norm} Choose a Riemannian metric as well as an orientation on $X_{reg}$. For each nondegenerate $p$-chain $c : M \to X_{reg}$, let $\mu_c$ be the induced volume $p$-form on $c(M)$; namely, $\mu_c$ is a unique $p$-form in $\Omega(X_{reg})$ such that, at each regular point of $c(M)$, we have:
$$\mu(v_1, \cdots, v_p) = 1$$
for an ordered orthonormal basis $v_i$ representing the orientation on $c(M)$. In short, it is the determinant viewed as a multi-skew-linear functional.

Then the \emph{volume} of $c$ is $\int_c \mu_c$, which is a positive real number. Since $c \mapsto \vol(c)$ is clearly continuous, we define $\vol : \Map_{PA}(\Delta_p, X) \to \mathbb{R}_{\ge 0}$ by continuous extension. The \emph{boundary volume} of $c$ we shall mean the sum of the volumes of all simplexes in $\partial c$.

For the standard $p$-simplex $\Delta_p$, we then let
$$X_p = \Map_{red-PA}(\Delta_p, X),$$
the set of all PA-maps $c : \Delta_p \to X$ with volume and boundary volume both $\le 1$; intuitively, they are the reduced $p$-simplexes on $X$. In particular, $X_0 = X$. Then for an arbitrary map $\varphi : X_p \to \mathbb{C}$, let
$$\rho(\varphi) = \sup_{c \in X_p} |\varphi(c)|,$$
which may or may not be finite. But if $X$ is compact and $\varphi$ is continuous, then it is finite (exercise).
\end{remark}

A current is then a linear functional on $\Omega_{PA}(X)$ that is continuous with respect to each compact subset of $X$ as follows. As usual, we write $\Gamma_S(U, F)$ for the space of sections of $F$ on $U$ with support contained in a set $S$.

\begin{definition}[currents]\label{def:currents} Let $X$ be a PA-space. Assume it has pure dimension $n$.

A \emph{$p$-current} on $X$ is a linear functional on $\Omega_{PA, \, cpt}^{n-p}(X)$ that is continuous with respect to $\rho$ on the subspace $\Gamma_K(X, \Omega_{PA}^{n-p})$ for each compact subset $K \subset X$.

We let $\Omega'^{\, p}_{PA}(X)$ denote the space of $p$-currents on $X$, although strictly speaking it is not a dual. Succinctly,
$$\Omega'^{\, p}_{PA}(X) = \varprojlim_K \Gamma_K(X, \Omega_{PA}^{n-p})'$$
where the superscript $'$ on the right means the continuous dual. Indeed, each element of the above limit is a family $u_K$ parametrized by compact sets $K$ such that $u_K \mapsto u_L$ under the restriction $\Gamma_K(X, \Omega_{PA}^{n-p})' \to \Gamma_L(X, \Omega_{PA}^{n-p})'$ for $L \subset K$.
\end{definition}

\begin{remark} Our definition of currents is somewhat inspired by \cite{Ambrosio} and the notion is actually analogous to normal currents in geometric measure theory.
\end{remark}

Just like $\Omega_{PA}$, we have:

\begin{remark}\label{differential for currents} By a partition of unity, it is easy to see that $\Omega'_{PA}$ is a sheaf. Also, it is an $\mathcal{O}_X$-module in the following way: for $f$ in $\mathcal{O}(U)$ and $u$ in $\Omega'_{PA}(U)$, $fu$ is defined by
$$\langle f u, \varphi \rangle = \langle u, f \varphi \rangle.$$
By Corollary \ref{Omega soft}, $\Omega'_{PA}$ is then a soft sheaf.

We define the differential $d'$ on $p$-currents $u$ by:
$$\langle d' u, \varphi \rangle = (-1)^{p + 1} \langle u, d \varphi \rangle.$$
Note $d'u$ is a current since $d : \Gamma_K(X, \Omega_{PA}) \to \Gamma_K(X, \Omega_{PA})$ is continuous in $\rho$.

Also, since each element of $\Omega^p_{PA}(X)$ can be paired with chains in $C_p(X)$, we have: $$C_p(X) \subset \Omega'^{\, n-p}_{PA}(X)$$
in such a way $d'$ restricts to $(-1)^{p+1}\partial$.
\end{remark}

The next proposition implies that we could have used a different space of test forms in the definition of currents.

\begin{prop}\label{prop:dense test forms} The completion $\overline{\Gamma_K(X, \Omega^p)}$ coincides with the subspace $\{ \varphi \mid \rho(\varphi) < \infty \}$ of $\Gamma_K(X, C_p')$.

In particular, in Definition \ref{def:currents}, we could have used $\Omega^{n-p}$ or sections of $C_{n-p}'$ with finite norms instead.
\end{prop}
\begin{proof} It is clear that the subspace $\{ \varphi \mid \rho(\varphi) < \infty \}$ of $\Gamma_K(X, C_p')$ is complete. Thus, the problem is to show $\Gamma_K(X, \Omega^p)$ is dense in the former. To show the denseness, by the Hahn--Banach theorem, it is enough to show the restriction map $\{ \varphi \mid \rho(\varphi) < \infty \}' \to \Gamma_K(X, \Omega^p)'$ is injective.

The proof is now by induction on $p$. If $p = 0$, then we note $C_0'(X) = \mathcal{C}^0(X)$, where $\mathcal{C}^0$ is the sheaf of continuous functions on $X$. Let $u$ be a continuous linear functional on $\Gamma_K(X, \mathcal{C}^0)$ that is zero on $\Gamma_K(X, \mathcal{O})$. Using a sequence of cutoff functions, we can assume $X = K$ and is contained in $\mathbb{R}^m$. Given a continuous function $\varphi$ on $X$, by the Stone--Weierstrass theorem, we can then find a sequence of polynomial functions $\varphi_j$ that approaches $\varphi$ as $j \to \infty$ in $\rho$. Then, since $u$ is continuous, $0 = \langle u, \varphi_j \rangle \to \langle u, \varphi \rangle$ as $j \to \infty$. Thus, $\langle u, \varphi \rangle = 0$.

Next, assume the assertion holds for each $q < p$. By Lemma \ref{C(U) dense}, we can assume $X = X_{reg}$. Then we can assume $X$ is an arbitrary small open subset of $\mathbb{R}^n$ and thus we can assume $X = Y \times (0, 1)$ for some open $Y \subset \mathbb{R}^{n-1}$. Then $C_p(X) = \oplus_q C_q(Y) \otimes C_{p-q}((0, 1))$; see \cite[\S 8.2.]{Kontsevich}. 

Let $u$ be a continuous linear functional on $\Gamma_K(X, C_p')$ that is zero on $\Gamma_K(X, \Omega^p)$. By inductive hypothesis with $Y$ in place of $X$, we have $u = 0$ on $C'_q(Y)$; thus on $(C_q(Y) \otimes C_{p-q}((0, 1)))'$ for $q < p$. Hence, the proof reduces to the case $X = [0, 1]$. Let $l$ be in $C'_1(I), \, I = [0, 1]$ and then define the function $f$ on $I$ by $f(x) = \langle l, [0, x])$. Note $f$ is continuous on $I$. Then, as above, we can find a sequence of polynomial functions $g_j$ that approaches $f$ on $I$. Then
$$\langle dg_i, [a, b] \rangle = g_i(b) - g_i(a) \to f(b) - f(a) = \langle l, [a, b] \rangle.$$
That is, $dg_j \to l$. Thus, $0 = \langle u, dg_j \rangle \to \langle u, l \rangle$ and so $\langle u, l \rangle = 0$.
\end{proof}

\begin{convention} Because of the above Proposition \ref{prop:dense test forms}, we will often write $\Omega'$ instead of $\Omega_{PA}'$.
\end{convention}

We also use the following lemma frequently. In the notations of Remark \ref{sup norm}, let
$$\rho_K(\varphi) = \sup \{ | \langle \varphi, c \rangle | \mid c \in X_p, \, \supp(c) \subset K \}.$$
Those seminorms $\rho_K$ then define a topology on $\Omega^{n-p}_{PA}(X)$; namely, the neighborhood base at zero consists of $\{ \varphi \mid \rho_K(\varphi) < r \}$ for all $K$ and real numbers $r > 0$ \cite[\S 4.1.1.]{geometric_measure}.

\begin{lemma}\label{compact continuous dual} If $X$ is compact, $\Omega'^{\, p}_{cpt}(X)$ coincides with the continuous dual of $\Omega^{n-p}_{PA}(X)$.
\end{lemma}
\begin{proof} For each compact subset $K \subset X$, consider
$$i_K : \Gamma_K(X, \Omega_{PA}^{n-p}) \hookrightarrow \Omega_{PA}^{n-p}(X).$$
It is continuous, since $\rho(\varphi) \ge \rho_L(\varphi)$ for each compact subset $L \subset X$. Thus, we get $\Omega_{PA}^{n-p}(X)' \to \Gamma_K(X, \Omega^{n-p})'$ for their continuous duals. Then taking the limit of them over all $K$, we get:
$$\Omega_{PA}^{n-p}(X)' \to \Omega'^{\, p}(X).$$
We claim that the above map is injective with the image equal to $\Omega'^{\, p}_{cpt}(X)$.

For injectivity, suppose $u \mapsto 0$. For $\varphi$ in $\Omega_{PA}^{n-p}(X)$, we can write
$$\langle u, \varphi \rangle = \langle u, \psi \varphi \rangle + \langle u, (1 - \psi) \varphi \rangle$$
where $\psi$ is in $\mathcal{O}_{cpt}(X)$. Here, the first term on the right is zero since $u$ is zero on $\Omega_{PA, \, cpt}^{n-p}(X)$. Also, since $u$ is continuous, for some $K$, we have $|\langle u, \cdot \rangle| / \rho_K$ is bounded on $\Omega_{PA}^{n-p}(X)$. In particular, if $\psi = 1$ on some neighborhood of $K$, then $\rho_K((1 - \psi) \varphi) = 0$ and so $\langle u, (1 - \psi) \varphi \rangle = 0$. Hence, $u = 0$.

As for the image, first we see that each $u$ in the image has compact support by the argument with continuity as above. Conversely, if $u$ is in $\Omega'^{\, p}_{cpt}(X)$, then define the linear functional $\widetilde{u}$ on $\Omega_{PA}^{n-p}(X)$ by $\langle \widetilde{u}, \varphi \rangle = \langle u, \psi \, \varphi \rangle$ where $\psi = 1$ on a neighborhood of the support of $u$ and has compact support. Then $\widetilde{u}$ maps to $u$.
\end{proof}

The next lemma gives a basic example of a current. Recall we can integrate a PA-form on an oriented PA-manifold by Lemma \ref{the integration of a PA-form}.

\begin{lemma}\label{lem:PA gives a current} If $X_{reg}$ is oriented; i.e., orientable with a choice of orientation, then for each $\omega$ in $\Omega^p_{L^{1, \, loc}}(X)$,
$$\varphi \mapsto \int_{X_{reg}} \omega \wedge \varphi$$
is a $p$-current. Depending on an orientation, we get in this way the sheaf monomorphism:
$$\epsilon : \Omega^p_{L^{1, \, loc}} \hookrightarrow \Omega'^{\, p}.$$

It is *not* a chain map in general; see Remark \ref{residue}.
\end{lemma}
\begin{proof}

We must check the well-defined-ness (i.e., the integral converges) and the continuity for all $\varphi$ with support contained in some fixed compact subset $K$ of $X$. Without loss of generality, we can assume $X$ is compact. By the sheaf property, since $X$ is compact, we only need to do this checking for some neighborhood of each point in $X$.

Thus, let $U$ be a neighborhood of a point in $X$. Shrinking $U$ and by linearity, we can assume $\omega = f_0 \, d f_1 \wedge \cdots \wedge d f_p$ on $U \cap X'$ for some open dense subset $X' \subset X_{reg}$ such that each $f_i$ is smooth on $X'$. Also, for a moment, assume $\varphi = \varphi_0 \, d \varphi_1 \wedge \cdots \wedge d \varphi_{n-p}$ for some $\varphi_i$'s in $\mathcal{O}(U)$ such that each $\varphi_i$ is smooth on $U \cap X'$.

Let $F = (f_0, \dots, f_p, \varphi_0, \dots, \varphi_{n-p}) : X \to \mathbb{R}^{n+2}$. Then 
$$\int_U \omega \wedge \varphi = \int_U F^*(x_0 \, dx \wedge y_0 \, dy)$$
where we wrote $dx = d x_1 \wedge \cdots \wedge dx_p$ and $dy = d y_1 \wedge \cdots \wedge dy_{n-p}$.

Note that, at each point in $U$, $d f_1 \wedge \cdots \wedge d f_p \wedge d \varphi_1 \wedge \cdots \wedge d \varphi_{n-p}$ vanishes if and only if $d f_1, \dots, df_p, d\varphi_1, \dots, d \varphi_{n-p}$ are linearly dependent. Thus, shrinking $U$ without changing the above integral, we can assume those differentials are linearly independent. Then, by the inverse function theorem, $\pi \circ F$ is an immersion on $U \cap X'$ and thus locally injective there, where $\pi(x_0, \dots, x_p, y_0, \dots, y_{n-p}) = (x_1, \dots, x_p, y_1, \dots, y_{n-p})$. In particular, $F$ is locally injective on $U \cap X'$.

Shrinking $U$, we can therefore assume $F$ is injective on $U \cap X'$. Shrinking $U$ further, we can write $F(U) = V \times W$ for some open subsets $V \subset \mathbb{R}^{p+1}, \, W \subset \mathbb{R}^{n-p+1}$. Then, by Fubini's theorem, the above integral equals
$$\int_{F(U)} x_0 \, dx \wedge y_0 \, dy = \int_V \left( \int_W y_0 \, dy \right) x_0 \, dx.$$
Here $\int_W y_0 \, dy = \int_{F^{-1}(W)} \varphi$ since $F$ is invertible on $U \cap X_{reg}$. Then
$$\left| \int_U \omega \wedge \varphi \right| \le \int_V \left| \int_{F^{-1}(W)} \varphi \right| |x_0| \, dx,$$
Now, it is clear that $\operatorname{span} \{ \varphi_0 \, d \varphi_1 \wedge \cdots \wedge d\varphi_{n - p} \mid \varphi_i \in \mathcal{O}(U), \varphi_i\textrm{ smooth on } U \cap X' \}$ is dense in $\Omega^{n-p}(U)$. Hence, the above estimate also holds for all $\varphi$ in $\Omega^{n-p}(U)$. The claimed continuity then follows.
\end{proof}

If $M$ is a compact oriented $n$-manifold, Stokes' formula implies $\int_M d\varphi = 0$ for a smooth $(n-1)$-form $\varphi$. This fact still holds more generally as follows.

\begin{lemma}\label{lem:integration map} If $X_{reg}$ is oriented and if $\operatorname{codim}(X - X_{reg}) \ge 2$, then, for each $\varphi$ in $\Omega^n_{cpt}(X)$,
$$\int_{X_{reg}} d\varphi = 0.$$
\end{lemma}
\begin{proof}

Let $U_{\epsilon} = X_{reg} \cap V_{\epsilon}$ where $V_{\epsilon}$ is a neighborhood of $X - X_{reg}$ such that the boundary $\partial U_{\epsilon}$ is piecewise-$\mathcal{C}^1$-smooth and $U_{\epsilon}$ approaches the empty set pointwise in the sense that for each $x \in X_{reg}$, $x \not\in U_{\epsilon}$ for all small enough $\epsilon$. This is possible because $X - X_{reg}$ is definable.
Then
$$\int_{X_{reg}} d\varphi = \lim_{\epsilon \to 0} \int_{X_{reg} - U_{\epsilon}} d\varphi,$$
since $\int_{U_{\epsilon}} d\varphi \to 0$ by Lebesgue's dominated convergence theorem as $d \varphi$ is integrable. Then, by Stokes' formula (the $\mathcal{C}^1$-version of Lemma \ref{Stokes}), we have
$$\int_{X_{reg} - U_{\epsilon}} d\varphi = -\int_{\partial U_{\epsilon}} \varphi,$$
which goes to zero as $\epsilon \to 0$ since $\varphi/\mu$ is bounded, $\mu$ the volume form on $\partial U_{\epsilon}$, and $\vol(\partial U_{\epsilon}) \to 0$.
\end{proof}

\begin{example}[cf. \cite{arithmetic} \S 1.1.2.] Let $X$ be a complex algebraic variety of dimension $n$. Then
$$\varphi \mapsto \int_{X_{reg}} \varphi$$
defines a $2n$-current $\delta_X$ such that $d' \delta_X = 0$. Indeed, the singular locus has complex codimension $\ge 1$; thus, real codimesion $\ge 2$.
\end{example}

\begin{remark}[residue]\label{residue} Assuming $X_{reg}$ is oriented, for each PA-form $\omega$ of type $p$, we have
\begin{align*}
\langle d \omega, \varphi \rangle &= \int_{X_{reg}} d \omega \wedge \varphi \\
&= \int_{X_{reg}} d(\omega \wedge \varphi) + \langle d' \omega, \varphi \rangle.
\end{align*}
Further assume $\operatorname{codim}(X - X_{reg}) \ge 2$. Then, by Lemma \ref{lem:integration map}, the first term on the right-most side vanishes. Thus,
$$\langle d \omega, \varphi \rangle = \langle d' \omega, \varphi \rangle.$$
That is, $d = d'$ on $\Omega_{PA}(X)$, meaning $\epsilon$ in Lemma \ref{lem:PA gives a current} is a chain map.

In general, however, $d' - d$ is not zero and we call that difference the \emph{residue map}, adapting \cite[Ch. 3., \S 1.]{griffiths_harris}.

The residue may also not vanish on a space larger than $\Omega_{PA}(X)$. For example, take $\omega = \frac{dz}{z}$ on $X = \mathbb{C}$, which is an integrable form if not a continuous form. Then, for $\varphi$ in $\mathcal{O}_{cpt}(X)$,
\begin{align*}
\langle d' \omega, \varphi \rangle &= \langle \omega, d\varphi \rangle = \int \frac{dz \wedge d\varphi}{z} =2\pi i \, \varphi(0) \\
&= 2\pi i \langle \delta_0, \varphi \rangle.
\end{align*}
On the other hand, $d \omega$ is zero as a section of $\Omega^1_{L^{loc, 1}}$ and so
$$(d' - d) \frac{dz}{z} = 2\pi i \, \delta_0,$$
which is to say the residue of $dz/z$ is $2\pi i \, \delta_0$.
\end{remark}

\begin{remark}[Mayer–Vietoris sequence] In \cite[§ 8.3.]{Kontsevich}, it is claimed that the Mayer–Vietoris short exact sequence holds on locally-closed PA-spaces for the sheaf $\Omega_{min}$ of minimal forms (meaning PA-forms with $r = 0$). However, that means $\Omega_{min}$ is in particular flasque, which, to us, seems nonsensical; since it means, for instance, $1/t$ on $\mathbb{R} - 0$ extends to some minimal form on $\mathbb{R}$. Also, the Mayer–Vietoris sequence in the form in loc. cit. would be contradictory to the results like Proposition \ref{PA approximation} or rather its corollary.
\end{remark}

On the other hand, earlier in Lemma \ref{chain decomposes}, we noted that PA-chains have the Mayer–Vietoris property. This property does generalize to currents. To show that, we use the following lemma, known as Serre's duality lemma \cite[\S 10., Lemme 1.]{Serre} or the closed range theorem.

\begin{lemma}\label{Serre lemmma} Let $E \overset{f}\to F \overset{g}\to G$ be a complex of (Hausdorff) locally convex topological vector spaces such that $g : F \to \im(g)$ is an open mapping.

Then the above complex is exact if and only if its continuous dual $G' \overset{g'}\to F' \overset{f'}\to E'$ is exact and $\im(f)$ is closed.
\end{lemma}
\begin{proof} Let $\im(f)^{\bot}$ denote the annihilator of $\im(f)$ inside $\ker(g)'$. Then we have the linear map
$$\varphi : \ker(f') \to \im(f)^{\bot}, \, \alpha \mapsto \alpha|_{\ker(g)}$$
since $f'(\alpha) = 0$ means $\alpha = 0$ on $\im(f)$. By the Hahn-Banach theorem, this map $\varphi$ is surjective. The kernel of $\varphi$ is exactly $\ker(g)^{\bot}$. Next, we factorize $g$ as $g = i \circ g_0 \circ p$:
$$F \overset{p}\to F/\ker(g) \overset{g_0}\to \im(g) \overset{i}\hookrightarrow G.$$
Then $g' = p' \circ g_0' \circ i'$. By the Hahn--Banach theorem, $i'$ is surjective. Since $g$ is an open mapping onto the image, $g_0$ is a topological isomorphism and thus $g_0'$ is an isomorphism. So, $\im(g') = \im(p')$.

In general, if $q : V \to V/W$ is a quotient map with closed $W$, then $q'$ is injective with image $W^{\bot}$. Thus, taking $q = p$, we get
$$\im(g') = \im(p') = \ker(g)^{\bot} = \ker(\varphi).$$
Hence, $\varphi$ induces:
$$\ker(f')/\im(g') \simeq \im(f)^{\bot} \simeq \left(\ker(g)/\overline{\im(f)} \right)'.$$
From this, the assertion follows.

Remark: in the classical formulation, we also have $\im(f')$ closed $\Rightarrow$ $\im(f)$ closed. But the standard proof of this, e.g., \cite[Theorem 4.14.]{Rudin}, relies on the open mapping theorem and we do not know how to avoid it.
\end{proof}

\begin{prop}[Mayer–Vietoris sequence]\label{Mayer–Vietoris} Let $X = Y_1 \cup Y_2$ be a union of closed PA-spaces. Then there is a natural short exact sequence
$$0 \to \Omega'(Y_1 \cap Y_2) \overset{k_{1 *}, \, -k_{2 *}}\to \Omega'(Y_1) \oplus \Omega'(Y_2) \overset{i_{1*} + i_{2*}} \to \Omega'(X) \to 0$$
where $i_j : Y_j \hookrightarrow X$ and $k_j : Y_1 \cap Y_2 \hookrightarrow Y_j$.
\end{prop}
\begin{proof}
Lemma \ref{chain decomposes} says we have the split exact sequence:
$$0 \rightarrow C(X) \rightarrow C(Y_1)\oplus C(Y_2) \rightarrow C(Y_1\cap Y_2) \rightarrow 0.$$
Here, the section is given by $C(Y_1)\oplus C(Y_2) \rightarrow C(X), \, (c_1,c_2) \mapsto c_1 + c_2 - c_1|_{Y_1\cap Y_2}$. Thus, by Lemma \ref{Serre lemmma}, we have:
$$0 \to \Gamma_K(X, C')' \to \oplus_{j=1}^2 \Gamma_K(Y_j, C')' \to \Gamma_K(Y_1 \cap Y_2, C')' \to 0.$$
By the Hahn--Banach theorem, the restriction $\Gamma_K(X, C')' \to \Gamma_{L}(X, C')'$ is surjective for $L \subset K$. In particular, the Mittag-Leffler condition holds for the projective system over compact subsets \cite[Example 9.1]{Hart}. Hence, taking the limit over $K$, we get the claimed exact sequence.
\end{proof}

\begin{corollary}\label{cor:PA-form Mayer} For the notations as above, we have the exact sequence
$$0 \to \Omega_{PA}(X) \overset{i_j^*}\to \oplus_{j=1}^2 \Omega_{PA}(Y_j) \overset{k_1^*, -k_2^*}\to \Omega_{PA}(Y_1 \cap Y_2) \to 0.$$
In short, $\Omega_{PA}$ is a sheaf for the topology in Remark \ref{compact topology}.
\end{corollary}
\begin{proof} We only have to show the exactness at the middle; i.e., $\im(i_j^*) = \ker(k_1^*, -k_2^*)$. Since the only question is the existence in the sheaf axioms; i.e., the inclusion $\supset$, the question is local and thus we can assume $X$ is compact. Then, by Lemma \ref{compact subspace cofibration}, we have a retraction $r : U \to Y_j$ of $i_j$ from some neighborhood $U$ of $Y_j$ in $X$; i.e., $r \circ i_j = \operatorname{id}_{Y_j}$. Thus, $i_j^* \circ r^*= \operatorname{id}$. But this implies that $i_j^*$ is an open mapping onto the image with closed image. Similarly, $k_j^*$ is an open mapping onto the image with closed image. Hence, since the continuous dual of the above sequence is exact, by Lemma \ref{Serre lemmma}, we get the required inclusion.
\end{proof}

\begin{remark}[Poincar\'e lemma for forms with compact support]\label{cpt Poincare} Recall that we proved the Poincar\'e lemma by establishing the formula
$$\varphi - \varphi|_0 = dJ \varphi + J d \varphi.$$
This does not directly give the lemma for compactly-supported forms since $J \varphi$ may not have compact support. We shall therefore modify $J$ as follows.

Let $U \subset X$ be a contractible open subset, $K \subset U$ a compact subset and by Corollary \ref{col:PA-homotopy}, $h : U \times I \to U$ a PA-homotopy such that $h_1$ is the identity and $h_0$ a constant map. Let $h^x$ denote the map $U \to \operatorname{Map}(I, U)$ that corresponds to $h$ by adjunction; namely, $h^x(t) = h(x, t)$.

Let $\varphi$ be a closed PA-form of type $p \ge 1$ on $U$ with support in $K$. Then $h^*_0 \varphi = 0$ for dimension reason and then, by the relative Stokes formula (Proposition \ref{FTC}), we have
$$d \int_{h^x} \varphi = x^* \varphi = \varphi,$$
where $x = h_1$ is the identity map, which we think of as a variable. It also shows that if $x$ varies outside $K$, then the right-hand side is zero; i.e., $\int_{h^x} \varphi$ is a closed form when it is restricted to $U - K$. We also note that if we vary $x$ outside $K$, then $\int_{h^x} \varphi$ does not change since $\varphi$ is supported on $K$.

Now, define the operator $J$ by
$$J \varphi = \int_{h^x} \varphi - \int_{h^{x_0}} \varphi$$
for some or any $x_0$ in $U - K$. Then since the second term on the right is closed, $d J \varphi = \varphi$. Also, the support of $J \varphi$ lies in $K$.
\end{remark}

Finally, we can state and prove: (note $U \mapsto \operatorname{H}^n_{cpt}(U)^*$ is a presheaf).

\begin{prop}[Poincar\'e lemma for currents]\label{prop:Poincare for currents} Let $\mathbb{C}_{Po}$ denote the sheaf on $X$ associated to the presheaf $U \mapsto \operatorname{H}^n_{cpt}(U)^*$ where $*$ refers to a vector space dual. Then, on $X$,
$$0 \to \mathbb{C}_{Po} \overset{\epsilon_2}\to \Omega'^{\, 0} \overset{d'}\to \Omega'^{\, 1} \overset{d'}\to \cdots \overset{d}\to \Omega'^{\, n} \to 0$$
is exact for some $\epsilon_2$, which may differ from $\epsilon$ in Lemma \ref{lem:PA gives a current}.
\end{prop}
\begin{proof} By Remark \ref{cpt Poincare}, for a contractible open subset $U \subset X$, we have the exact sequence
$$0 \to \Omega^0_{PA, \, cpt}(U) \overset{d}\to \cdots \overset{d}\to \Omega^n_{PA, \, cpt}(U) \overset{q}\to \operatorname{H}^n_{cpt}(U) \to 0$$
where $q$ is the quotient map. Note the very first $d$ is injective by Lemma \ref{poincare}, saying in particular that $df = 0$ implies $f$ is locally constant for a PA-$0$-form $f$.

For each compact subset $K \subset U$, by an explicit construction in Remark \ref{cpt Poincare}, we see the corresponding sequence in terms of $\Gamma_K(U, \Omega^p_{PA})$ is exact and is in fact split exact. Also, $q$ on $\Gamma_K(U, \Omega_{PA}^n)$ is an open mapping since it is a quotient map of topological groups. By repeating the classical argument, we see that $\operatorname{H}^n_{cpt}(U)$ has finite dimension (cf. \cite[Proposition 5.3.2.]{DifferentialForms}); in particular, it is Hausdorff.

Hence, by applying Lemma \ref{Serre lemmma} for each $K$ and then taking the projective limit (which is exact by the Mittag-Leffler condition), we get the asserted exact sequence.
\end{proof}

\begin{remark} From the exact sequence with currents above, it \emph{does not} follow the Poincar\'e lemma holds for $\Omega(X)$, since $d$ on $\Omega(X)$ need not be an open mapping onto the image (so Lemma \ref{Serre lemmma} does not apply).
\end{remark}

Next, to understand $\mathbb{C}_{Po}$, first we note:

\begin{prop}\label{direct image quasi-isomorphism} Let $\pi : \mathbb{R}^m \to \mathbb{R}^{m-r}$ be the projection and
$$\pi_* : \Gamma_{cpt}(\mathbb{R}^m, \Omega_{\mathbb{R}^m, PA}) \to \Gamma_{cpt}(\mathbb{R}^{m-r}, \Omega^{* - r}_{\mathbb{R}^{m-r}, PA})$$
be the relative integration along $\pi$. Then $\pi_*$ is a quasi-isomorphism; i.e., induces the isomorphism on cohomology.
\end{prop}
\begin{proof} By Fubini's theorem, $\pi_*$ is the composition of the direct images along $\mathbb{R}^m \to \mathbb{R}^{m-1} \to \cdots \to \mathbb{R}^{m - r}$. Since the composition of quasi-isomorphisms is a quasi-isomorphism, it is thus enough to consider the case $r = 1$.

Let $F$ denote the $\mathbb{R}$ that is a fiber of $\pi$. Then fix a one-form $b$ in $\Omega^1_{cpt}(F)$ such that $\int b = 1$. Then let $B(t) = \int^t_{-\infty} b$.

Let $V = \Gamma_{cpt}(\mathbb{R}^m, \Omega_{\mathbb{R}^m, PA})$. Following \cite[Proposition 4.6.]{DifferentialForms}, define the operator on $V$
$$T = \pi_{t, *} - B \pi^*\pi_*$$
where $t$ is a coordinate \emph{function} on $F$, $\pi_{t, *} = \int_{-\infty}^t$ and $B$ is viewed as a function on $\mathbb{R}^m$ via $B(x, t) = B(t)$. Then, since $d \pi_* = -\pi_* d$ on $V$, by Proposition \ref{FTC} or more precisely a PA-form analog of that, we have:
\begin{align*}
dT &= |_t - \int_{-\infty}^t d - dB \wedge \pi^*\pi_* - B \pi^*d \pi_* \\
&= \operatorname{id} - Td - b \wedge \pi^* \pi_*.
\end{align*}
That is, $dT + T d = \operatorname{id} - b \wedge \pi_*$. On the other hand, $\pi_* \circ (b \wedge \pi^*) = \operatorname{id}$ on $\Omega_{PA}(X)$.
\end{proof}

\begin{remark}
By Proposition \ref{direct image quasi-isomorphism}, for each connected open subset $U \subset X_{reg}$ that is isomorphic to an open subset of $\mathbb{R}^n$, we see the integration over $U$ induces the isomorphism
$$\operatorname{H}^n_{cpt}(U) = \mathbb{C}$$
(or alternatively, we can see it by the Poincar\'e duality). Indeed, the proposition, or more precisely its proof implies we have the exact sequence
$$\Gamma_K(\mathbb{R}^n, \Omega_{PA}^{n-1}) \to \Gamma_K(\mathbb{R}^n, \Omega_{PA}^n) \overset{\pi_*}\to \mathbb{C} \to 0$$
for each compact subset $K \subset \mathbb{R}^n$. But since $U$ is isomorphic to an open subset of $\mathbb{R}^n$, $\Gamma_K(U, \Omega_{PA}) = \Gamma_K(\mathbb{R}^n, \Omega_{PA})$. Thus, the above implies $\pi_* : \operatorname{H}^n_{cpt}(U) \simeq \mathbb{C}.$

But if $U \subset X$ is a connected open subset, we could have $\operatorname{H}^n_{cpt}(U) \neq \mathbb{C}.$ Indeed, let $X = [0, 1]$. Then $\operatorname{H}^1_{cpt}(U) = 0$ for each connected neighborhood $U$ of $0$ as follows. For each $[\varphi]$ in $\operatorname{H}^1_{cpt}(U)$, let $\eta = \int^x_0 \varphi - \int_0^1 \varphi$. Then we have $d \eta = \varphi$ and $\eta$ has compact support in $U$.

Later (Examples \ref{Poincare example}), we will see that even if $\operatorname{codim}(X - X_{reg}) \ge 2$, it can happen $\mathbb{C}_{Po} \ne \mathbb{C}$. Note this is consistent with the fact that $d'$ and $d$ agree on $\Omega_{PA}(X)$ when $\operatorname{codim}(X - X_{reg}) \ge 2$. Indeed, $\mathbb{C}_{Po}|_{X_{reg}} = \mathbb{C}|_{X_{reg}}$. Thus, a section of $\mathbb{C}_{Po}$ over a connected open subset that is not a constant function must be a piecewise-constant function; in particular, not continuous.


\end{remark}

We also record a few results we use in the next section. First, we have the following version of Fubini's theorem.

\begin{prop}[Fubini for currents]\label{Fubini current} Let $u$ be a $p$-current on $X$ with compact support and $Y$ a PA-space of dimension $m$.

Then for each $\varphi$ in $\Omega^{n-p}(X \times Y)$ and each $\psi$ in $\Omega^m(Y)$, if $\psi$ is integrable, we have:
$$\int_Y (y \mapsto \langle u, \varphi_y \rangle) \psi = \langle u, \pi_*(\psi \wedge \varphi) \rangle$$
where $\varphi_y$ denotes a form on $X$ given as $x \mapsto \varphi(x, y)$, $\pi : X \times Y \to X$ the projection and $\psi$ on the right is viewed as a form on $X \times Y$ in the natural way.
\end{prop}
\begin{proof} Let $v$ be a linear functional on $\Omega^{n -p}(X \times Y)$ defined by 
$$\langle v, \varphi \rangle = \int \langle u, \varphi_y \rangle \psi - \left\langle u, \int \psi \wedge \varphi_y \right\rangle.$$
It is clearly continuous. We shall show $v$ is zero. First, suppose $\varphi = \varphi_1 \varphi_2$ where $\varphi_1$ is in $\Omega^{n-p}(X)$ and $\varphi_2$ is in $\mathcal{O}(Y)$. Then it is trivial that $\langle v, \varphi \rangle = 0$.

In general, let $\psi = \psi_1 \psi_2$ such that $\int \psi_1 = 1 = \int \psi_2$. Then, in the topology on $\Omega(X \times Y)$, we have:
$$\varphi * \psi_t \to \varphi \textrm{ as } t \to 0$$ 
where $\psi_t = \psi(\cdot / t)/t^n$ for $t > 0$. On the other hand,
$$(\varphi * \psi_t)(x) = \prod \int \varphi(y) \, \psi_t(y - x) = \prod_1^2 \int \varphi(y_i) \psi_{i, t}(y_i - x_i)$$
where we wrote $x = (x_1, x_2)$ and similarly for $y$.
\end{proof}

For later use, we note the following.

\begin{prop}\label{prop:Cauchy complete} Let $u_i$ be a sequence of $p$-currents on $X$. If $u_i$ is Cauchy in the sense $\langle u_i, \varphi \rangle$ is Cauchy for each $\varphi$ in $\overline{\Omega^{n-p}_{cpt}}(X) = \varinjlim \overline{\Gamma_K(X, \Omega^{n-p})}$, the bar meaning completion, then $u_i$ has unique limit $u$ that is also a $p$-current.
\end{prop}
\begin{proof} By assumption, we can define a linear functional $u$ on $\Omega^{n-p}_{cpt}(X)$ by $\displaystyle \langle u, \varphi \rangle = \lim_{i \to \infty} \langle u_i, \varphi \rangle$. So, the problem is to show it is continuous on $\Gamma_K(X, \Omega^{n-p})$ for each compact subset $K \subset X$.

By the uniform boundedness principle, we get:
$$M := \sup_i \|u_i\|_{op} < \infty,$$
where $\|\cdot\|_{op}$ is the operator norm. Then, for each $i$ and each $\varphi$ in $\Gamma_K(X, \Omega^{n-p})$,
$$|\langle u_i, \varphi \rangle| \le M \, \rho(\varphi).$$
Taking the limit, we get $|\langle u, \varphi \rangle| \le M \, \rho(\varphi)$, which shows $u$ is continuous.
\end{proof}

\begin{prop}\label{denseness of test forms} $\Omega^p_{cpt}(X)$ is dense in $\Omega'^{\, p}(X)$ with respect to the weak-* topology.
\end{prop}
\begin{proof} By the Hahn--Banach theorem, it is enough to show the restriction map
$$\Omega'^{\, p}(X)' \to \Omega^p_{cpt}(X).$$
But by the reflexivity of a weak-* topology, this is trivial.
\end{proof}

Throughout the present section, some readers might wonder why we couldn't just work with the completions of spaces of test forms. It is because we usually lose exactness by the below.

\begin{remark} The completion functor $E \mapsto \overline{E}$ for normed spaces is not left-exact; i.e., there is an injective continuous linear map $f : E \to F$ such that $\overline{f} : \overline{E} \to \overline{F}$ is not injective \cite[Example 1.5.]{Varol}.
\end{remark}

\section{Hodge theorem}\label{sec:Hodge theorem}

The usual Hodge theory takes place on compact K\"ahler manifolds. Here, we develop a version of the theory that is valid on K\"ahler varieties. The basic idea is to limit ourselves to a filtration on harmonic forms on compact subsets (a somehow similar idea is already given at \cite[\S. 3.3.]{arithmetic}.)

\

Let $X$ be a PA-space.

\begin{convention}

Throughout the present section, we fix a defining cover $X_i$ of $X$; in particular, it comes with $X_i \hookrightarrow \mathbb{R}^{m_i}$. Also, we assume $X_{reg}$ comes with a Riemannian metric, is oriented and has pure dimension $n$. Furthermore, we assume $\codim(X - X_{reg}) \ge 2$ for simplicity (i.e., to avoid the complication discussed at Remark \ref{residue}).

We say a real-valued function $f$ on $X_{reg}$ is \textit{uniformly continuous} if for each $i$, the restriction $f|_{X_i \cap X_{reg}}$ is uniformly continuous with respect to the above $X_i \hookrightarrow \mathbb{R}^{m_i}$ or equivalently it extends to a continuous function on $X_i$.

And through the section, we shall then assume the above Riemannian metric is uniformly continuous in the sense that $x \mapsto (y, z)_x$ is uniformly continuous on $X_{reg}$ for each $y, z \in T(X_{reg})$. Without this uniform continuity assumption, there is a counterexample to the main result of this section (Example \ref{metric example 2}).

\end{convention}

\begin{construction}[pairing and Hodge star] Here, we shall construct a pairing and essentially equivalently the Hodge star operator.

Let $M = X_{reg}$. First, we repeat the usual smooth case. Through the Riemannian metric, we have the smoothly varying real inner products on $\Omega^1_{\mathcal{C}^{\infty}, x}$ for $x$ on $M$. Thus, we have the pairing
$$\langle, \rangle : \Omega^1_{\mathcal{C}^{\infty}}(M) \times \Omega^1_{\mathcal{C}^{\infty}}(M) \to \mathcal{C}^{\infty}(M).$$
It then extends to $\Omega^p_{\mathcal{C}^{\infty}}(M)$ by linearity and the requirement
$$\langle f, g \rangle = \det \langle f_i, g_j \rangle$$
for $f = f_1 \wedge \cdots \wedge f_p$ and $g = g_1 \wedge \cdots \wedge g_p$ \cite[Ch. 2., Exercise 13.]{Warner}.

Assume $M = X_{reg}$ has a volume form $\mu$. Then, e.g. as in the proof of \cite[Lemma 3.6.]{Faltings}, we define the operator $* : \Omega_{\mathcal{C}^{\infty}}^p(M) \to \Omega_{\mathcal{C}^{\infty}}^{n-p}(M)$, the \emph{Hodge star operator}, by
$$\langle f, g \rangle \, \mu = f \wedge * g.$$
Note we have $** = (-1)^r$ on $\Omega^p_{\mathcal{C}^{\infty}}(M)$ for some integer $r$ depending on $p$.

By our convention, $\mu$ extends to a continuous volume form on $X$. Then we get
$$* : \Omega^p_{L^{1, loc}}(X) \to \Omega^{n-p}_{L^{1, loc}}(X)$$
by approximating locally integrable forms by smooth forms (the dominated convergence theorem applies since the volume form is locally integrable). Note here it is \emph{not} enough that we have a volume form on $X_{reg}$; cf. Example \ref{metric example}.

Now, the Hodge star $*$ \emph{does not restrict} to an endomorphism on $\Omega(X)$. Thus, we cannot define $*$ for currents by transpose like we did for the differential $d$ in Remark \ref{differential for currents}. Instead, we use continuous extension. Namely, we have:
$$* : \Omega^p(X) \to \Omega'^{\, n-p}(X).$$
It is continuous for the current topology (i.e., the weak-* topology). Indeed, suppose $f_j \to 0$ in $\Omega^p(X)$. Then for each $g$ in $\Omega^{n-p}(X)$, we have
$$g \wedge *f_j = \langle g, f_j \rangle \mu \to 0.$$
That implies $*f_j \to 0$.

Then, by the denseness of the space of test forms (Proposition \ref{denseness of test forms}), the above operator extends to 
$$* : \Omega'^{\, p}(X) \to \Omega'^{\, n-p}(X).$$
This completes the construction of the Hodge star operator.
\end{construction}

By the above pairing, we have the inner product on $\Omega_{cpt}(X)$ given by
$$(f, g) = \int_{X_{reg}} \langle f, \overline{g}\rangle \, \mu$$
where the bar means complex conjugation. Let $\Omega_{L^2}(X)$ denote the completion of $\Omega_{cpt}(X)$ with respect to the $L^2$-norm defined by the inner product. Note $\Omega_{L^2}(X) \subset \Omega'(X)$ since $\Omega'(X)$ is complete in the sense in Proposition \ref{prop:Cauchy complete}. In this definition, we could have used $X_{reg}$ instead of $X$ because $\Omega_{\mathcal{C}^{\infty}, cpt}(X_{reg})$ is dense in $\Omega_{L^2}(X)$.

With the Hodge star operator $*$, we define
$$d^* = (-1)^{p+1 + r_p} *d*$$
on $\Omega'^{\, p}(X)$ where $r_p$ is an integer such that $** = (-1)^{r_p}$ on smooth $p$-forms on $X_{reg}$. Then for $f, g$ in $\Omega^p(X)$, if $f$ has compact support in $X_{reg}$, we have, by Stokes' formula,
\begin{align*}
(f, d^* g) &= (-1)^{p+1} \int f \wedge (d*g) = \int df \wedge *g \\
&= (df, g).
\end{align*}
That is, $d^*$ is the adjoint of $d$ in the above sense and is called the \emph{formal adjoint} of $d$.

Finally, we introduce the Laplacian.
$$\Delta_d = d^* d + dd^*.$$
By the above discussion, it is formally self-adjoint.

Next, we try to address the nontrivial issue of the notion of differential operators.

\begin{remark}[differential operators; cf. \cite{Schapira} Ch. I., \S 1 and \cite{Warner} Remark 6.36]\label{remark: principal symbol} Assume $X \subset \mathbb{R}^m$ is a compact PA-space. We write $\widetilde{X}$ for the phase space to $X$ as introduced in Construction \ref{phase space construction}; i.e., each fiber of $\widetilde{X}$ is the Zariski cotangent space. Using the metric, we also have $\widetilde{X_{reg}} \subset X_{reg} \times (\mathbb{R}^m)^*.$

Let
$$\varphi : X_{reg} \to G(n, (\mathbb{R}^m)^*)$$
be the map to the real Grassmannian given by $\varphi(x) = \mathfrak{m}_{pol, \, x}/\mathfrak{m}_{pol, \, x}^2$. Then let
$$Q_X \subset \widetilde{X} \times \overline{\operatorname{im}(\varphi)}$$
be the closure of the set of all triples $(x, \xi, \varphi(x))$ consisting of $x \in X_{reg}, \xi \in \varphi(x)$. If $(x, \xi, V)$ is in $Q_X$, then, clearly, $\xi \in V$ and $V \subset \mathfrak{m}_{pol, \, x}/\mathfrak{m}_{pol, \, x}^2$.


Let $P : \Omega' \to \Omega'$ be a morphism of sheaves of vector spaces. Then, for each $\gamma = (a, \xi, V) \in Q_X$, if possible, define the linear map
$$p(\gamma): \wedge V \to \wedge V$$
by
$$p(\gamma)(dx^i) = (e^{-f} P(e^{f} dx^i))(a)$$
where $f$ is a section of $\mathcal{O}$ over some neighborhood of $x$ such that $(d f)(x) = 2\pi i \, \xi.$

We said ``if possible'' since the above $p$ may not be well-defined. For a lack of standard terminology, if the above $p$ is defined, $p(x, \xi, V)$ is a polynomial function in $\xi$ with coefficients in $\End(\wedge V) \simeq \End(\wedge \mathbb{R}^n)$ and those coefficients are continuous in $x$, then we shall call $P$ a \emph{classical differential operator with continuous coefficients} and $p$ the \emph{total symbol} of $P$.

For such $P$, we can then write
$$p(x, t\xi, V) = p_0(x, \xi, V) t^r + \cdots.$$
Then $r$ is called the \emph{order} of $P$ and $p_0$ the \emph{principal symbol} of $P$ denoted by $\sigma(P) = p_0$.

For example, if $P = d$, then the total symbol and the principal symbol are both $2\pi i \, \xi \wedge -$. If $P = *$ and $X$ is smooth, then the total and principal symbols are both $*$. If $X$ is not smooth, $*$ is not necessarily classical (not induced by an operator on $X$).

An \emph{elliptic operator} is a classical differential operator whose $p_0$ is continuous on $Q_X$ and $p_0(\gamma)$ is invertible for each $\gamma$ in $Q_X$.

(See also \S \ref{sec:differential operator} for the discussion of vector fields as differential operators.)
\end{remark}

Here is a simple example illustrating the situation in the above remark.

\begin{example}\label{Q_X example} Let $X = V(w^2 - z^3) \subset \mathbb{C}^2$. If $x$ is in $X_{reg} = X - 0$, then $\mathfrak{m}_{pol, x}/\mathfrak{m}_{pol, x}^2$ is spanned by either $\{ dz, d \bar{z} \}$ or $\{ dw, d \overline{w} \}$. Define the pairing by $\langle dz, d z\rangle = \langle d\bar z, d \bar z \rangle = 1, \langle dz, d\bar{z} \rangle = 0.$ Then $\mathfrak{m}_{pol, x}/\mathfrak{m}_{pol, x}^2$ is identified with $(\mathbb{R}^2)^* \times 0 \subset (\mathbb{R}^4)^*$. Now, if $(x, \xi, V)$ is in $Q_X$, then $V = (\mathbb{R}^2)^* \times 0$.
\end{example}

Like in the classical case, we shall prove the elliptic regularity (note smoothness is easy; the nontrivial part is square integrability). For that, we shall use a Sobolev space (an alternative is by pseudodifferential operators, which seems to be an overkill). To define that, we shall use the Fourier transformation. It is not clear how to define the Fourier transform of a $p$-current in general. Thus, we shall use:

\begin{lemma}\label{lem:current resolution} Assume $X$ is compact and fix some $X \hookrightarrow \mathbb{R}^m$. Let $x_i$'s be the standard coordinates on $\mathbb{R}^m$. Then we have:
$$\Omega'^{\, p} \hookrightarrow \bigoplus_{1 \le i_1 < \, \cdots \, < i_{n-p} \le m} \Omega'^{\, n}$$
given by $u \mapsto u_i$ where
$$\langle u_i, \varphi \rangle = \langle u, \varphi \, dx^i \rangle$$
and $dx^i = dx_{i_1} \wedge \cdots \wedge dx_{i_{n-p}}.$
\end{lemma}
\begin{proof} Note the $dx^i$'s span $\Omega_{pol}^p(U)$ over $\mathcal{O}_{pol}(U)$, an open subset $U \subset X$. Thus, we have the presentation:
$$\bigoplus_{i_1 < \, \cdots \, < i_{n-p}} \mathcal{O}_{pol} \overset{\pi}\to \Omega_{pol}^{n-p} \to 0$$
given by $f_i \mapsto \sum_i f_i \, dx^i$.

Now, we note that $\Omega_{pol}^p(K)$ is dense in $\Omega_{PA}^p(K)$ for each compact subset $K \subset U$. Indeed, we know $\Omega^p(K)$ is dense in $\Omega_{PA}^p(K)$ by Proposition \ref{prop:dense test forms}. On the other hand, by the Stone--Weierstrass theorem; that is, a version of it for continuous forms, $\Omega_{pol}^p(X)$ is dense in $\Omega^p(X)$.

Then, by Lemma \ref{Serre lemmma}, taking the dual, we get:
$$0 \to \Omega'^{\, p} \overset{\pi^*}\to \bigoplus_{i_1 < \cdots < i_{n-p}} \Omega'^{\, n}.$$


\end{proof}

\begin{construction}[Fourier transformation]\label{Fourier transformation} Let us be in the setup of the preceding Lemma \ref{lem:current resolution} as well as Remark \ref{remark: principal symbol}. Note each $0$-current on $X$ extends uniquely to a continuous linear functional on the completion $\overline{\mathcal{O}}(X)$.

Then, for each $\xi$ in $(\mathbb{R}^m)^*$, we define the \emph{Fourier coefficient $\widehat{u}(\xi)$ of $u$} for $\xi$:
$$\widehat{u}(\xi) = (\widehat{u}_i(\xi) \mid i_1 < \cdots < i_{n-p}) \in \mathbb{C}^{\binom{m}{n - p}}$$
where
$$\widehat{u}_i(\xi) = \langle u_i, e^{-2\pi i \, \xi|_X} \rangle.$$


\end{construction}

A key result is of course the inversion formula:

\begin{prop}[Fourier inversion formula]\label{Fourier inversion formula} Assume $m = n$. In the setup of Construction \ref{Fourier transformation}, $\widehat{\widehat{u}}$ can also be defined and $$u = \tau^*\widehat{\widehat{u}}$$
where $\tau : \mathbb{R}^m \to \mathbb{R}^m, \, x \mapsto -x$.
\end{prop}
\begin{proof} It is enough to prove the case when $p = n$.

There are at least two proofs of the inversion formula: the one using Fubini's theorem and the other more of a representation-theory flavor. Here we shall use the first one.

Let $g(x) = e^{-\pi |x|^2}$ be the standard Gaussian function on $\mathbb{R}^m$. Then we have $\widehat{g} = g$ using $\int_{-\infty}^{\infty} e^{-\pi(x + ia)^2} \, dx = 1$ for each real number $a$.

By Fubini's theorem for currents (Proposition \ref{Fubini current}), for $t > 0$, we have:
$$\int \langle u, e^{-2 \pi i \xi} \rangle \, e^{2 \pi i \xi(x)} \, g(t \xi) \, d\xi = \left\langle u, \int \, e^{-2 \pi i \xi(\cdot - x)} \, g(t \xi) \, d\xi \right \rangle.
$$
Now, the integration appearing on the right equals
$$g(t \cdot )^{\wedge}(\cdot - x) = 
\widehat{g}((\cdot - x)/t)/t^n = g((\cdot - x)/t)/t^n.$$
Thus, by Fubini again, we get
$$\left \langle \tau^* \widehat{\widehat{u} \, g(t \cdot)}, \varphi \right \rangle = \langle u, \varphi * g(\cdot /t) / t^n \rangle.$$
Note $\langle u, e^{-2\pi i \xi} \rangle$ is bounded by a constant only depending on $u$ (since $|e^{-2\pi i \xi}| = 1$). Thus, letting $t \to 0$ then gives the assertion.
\end{proof}

\begin{corollary}\label{cor:integrable Fourier} If $\widehat{u}$ is integrable, then $u$ is continuous; in fact, Lipschitz continuous in the sense
$$| u(x) - u(y)| \le 2\pi \|\widehat{u}\|_{L^1} |x - y|.$$
\end{corollary}
\begin{proof} Let $v(x) = \int \widehat{u} \, e^{2\pi i \langle x, - \rangle} d \, \xi$. Then we have
$$|v(x) - v(y)| = \left|\int \widehat{u} \, e^{2\pi i \langle x - y, - \rangle } d\, \xi \right| \le 2\pi \|\widehat{u}\|_{L^1} |x - y|.$$
Thus, $v$ is continuous; in particular, is a current and as currents, $v = \tau^*\widehat{\widehat{u}} = u$.
\end{proof}


\begin{remark}[Plancherel theorem]\label{remark:Plancherel} The Fourier inversion implies the Plancherel theorem, the unitarity of the Fourier transformation. Indeed, first note the Fourier transformation of $\Omega_{\mathcal{C}^{\infty}}(X_{reg})$ is in $\Omega_{L^2}(\mathbb{R}^m)$ with $\|\widehat{f}\| \le \|f\|$. Thus, by uniform continuity, it extends uniquely to the continuous operator
$$\widehat{\cdot} : \Omega_{L^2}(X) \to \Omega_{L^2}(\mathbb{R}^m).$$
It is then unitary onto the image since $\|\widehat{f}\| \le \|f\| = \|\widehat{\widehat{f}}\| \le \|\widehat{f}\|.$

\end{remark}

Using the Fourier transformation, we then have:

\begin{definition}[Sobolev space]\label{Sobolev space} For each real number $s$, if $X$ is compact and embedded in $\mathbb{R}^m$, let
$$\|u\|_{(s)} = \|(1+|\cdot|^2)^{s/2} \, \widehat{u}\|_{L^2}$$ be the Sobolev norm on $X$. Then we let $W^s(X) \subset \Omega'(X)$ be the subspace consisting of all $u$ such that $\|u\|_{(s)} < \infty$.

\end{definition}

\begin{prop}[Gårding's inequality; cf. \cite{Faltings} Theorem 3.2.]\label{Garding} Assume $X$ is compact and contained in $\mathbb{R}^m$.

Let $P : \Omega'(X) \to \Omega'(X)$ be a linear operator such that for each trivial $F \subset \widetilde{X}$, $P|_F$ is elliptic.

Then there exists a constant $c$ such that for each trivial $F \subset \widetilde{X}$,
$$c \, \|u\|_{(s+r)} \le \|u\|_{(s)} + \| P \, u \|_{(s)}$$
where the Sobolev norms are defined with respect to $F$.
\end{prop}
\begin{proof} We note that, by means of a partition of unity, it is enough to show the estimate for arbitrary small $X$. Translating $X$, we assume the origin $0$ lies in $X$.

Let $p$ be the total symbol of $P$. We write $p = p_0 + q$ where $p_0$ is the principal symbol. For the clarity of the argument, first we consider the case $p = p_0$.

For $\gamma = (x, \xi)$ in $F$ and $f = 2\pi i \, \xi$, since $P$ is homogeneous of degree $r$, if it is graded of degree $l$, we have:
\begin{align*}
(\widehat{Pu})_j(\xi) &= \langle P u, e^{-f} dx^j \rangle = (-1)^l \langle u, P(e^{-f} dx^j) \rangle = (-1)^l \langle u, p(-\xi) dx^j\rangle \\
&= \langle p(\xi) u, e^{-f} dx^j \rangle.
\end{align*}
Thus,
$$\|Pu\|_{(s)} = \left( \int |\langle p(\gamma) u, e^{-f} \rangle|^2(1+|\xi|^2)^s \, d\xi \right)^{1/2}$$
where, more precisely, $|\cdot|$ involves the summation over $j$ but we omit that for notational simplicity.

Let $\gamma_0 = (0, \xi)$. Since $P$ is elliptic on $F$, we have a constant $c$ \emph{independent of $F$} such that 
$$|\langle p(\gamma_0) u, e^{-f} \rangle |  \ge c|\xi|^r |\widehat{u}|.$$
Since the minimum value of $\displaystyle \frac{1 + t^{2r}}{(1 + t^2)^r}$ is $2^{1 - r}$, we have:
$$c \|u\|_{(s)}^2 + \int |\langle p(\gamma_0) u, e^{-f} \rangle|^2 (1+|\xi|^2)^s \, d\xi \ge 2^{-r} c \int |\widehat u|^2 (1+|\xi|^2)^{s+r} \, d\xi.$$
Also, shrinking $X$, we can have
$$\int |\langle (p(\gamma) - p(\gamma_0)) u, e^{-f} \rangle|^2 |\widehat u|^2(1+|\xi|^2)^s \, d\xi \le \|u\|_{(s)}.$$

Using $a + b \ge (a^2 + b^2)^{1/2}$, it follows
$$c \|u\|_{(s)} + \|P u\|_{(s)} \ge 2^{-r} c \|u\|_{(s+r)} - \|u\|_{(s)},$$
which is the assertion when $p = p_0$.

In general, as in the first part, we have:
$$\|P u\|_{(s)} + (1 + c)\|u\|_{(s)} \ge 2^{-r} c \|u\|_{(s+r)} - \left( \int |(q(\gamma) u)^{\wedge}|^2(1+|\xi|^2)^s \, d\xi \right)^{1/2}.$$
Now,
$$|(q(\gamma_0) u)^{\wedge}| = |q(\gamma_0) \widehat{u}| \le \|q(\gamma_0) \| |\widehat{u}|$$
since $q$ is continuous. Also, shrinking $X$, we can have
$$\int |(q(\gamma) - q(\gamma_0)) u)^{\wedge}|^2 |\widehat u|^2(1+|\xi|^2)^s \, d\xi \le \|u\|_{(s)}.$$
Hence,
$$\|u\|_{(s + r)} \le c' (\|P u\|_{(s)} + \|u\|_{(s)} + \|u\|_{(s + r - 1)})$$
for some constant $c'$. Now, the above inequality can be used to estimate $\|u\|_{(s + r - 1)}$ in the right-hand side. Thus, applying the above inequality recursively, we get the claimed inequality.
\end{proof}

Recall
$$\Delta_d = d^* d + d d^*.$$
As in the proof of the above proposition, we easily see the principal symbol of $\Delta_d$ is exactly $\xi^* \xi + \xi \xi^*$ where by $\xi$, we mean the left multiplication by $\xi$.

In our setup, since $\Delta_d$ is not classical, it is not elliptic. What we can say is that if $F \subset Q_X$ is trivial, then $\Delta_d|_F$ is elliptic. This follows from

\begin{lemma} For each point $(x, \xi, V)$ in $Q_X$, the operator $T = \xi^* \xi + \xi \xi^*$ is invertible, where by $\xi$, we mean the left multiplication by $\xi$.
\end{lemma}
\begin{proof} Since $\wedge V$ has finite dimension, it is enough to show $T$ is injective. Since $\xi \wedge \xi = 0$, we have the Koszul complex
$$0 \to \mathbb{C} \overset{\xi}\to \wedge^1 V \overset{\xi}\to \wedge^2 V \to \cdots.$$
In fact, we know from linear algebra that the above complex is exact.

On the other hand, we have the well-defined injection
$$\operatorname{ker}(T) \hookrightarrow \textrm{the cohomology of }(\wedge V, \xi)$$
given by $u \mapsto [u]$. (This is essentially a part of the Hodge theorem in the algebraic setup; cf. Theorem \ref{L^2 Hodge}.) Indeed, first we have:
$$\ker(T) = \ker(\xi) \cap \ker(\xi^*)$$
since
$$(Tu, u) = \|\xi \,u\|^2 + \|\xi^* \, u\|^2.$$
Thus, $u \mapsto [u]$ is well-defined. Also, if $[u] = 0$, then $u$ is in $\operatorname{im}(\xi) \subset \operatorname{ker}(\xi^*)^{\bot}$; thus, $u$ is in $\operatorname{ker}(\xi^*) \cap \operatorname{ker}(\xi^*)^{\bot} = 0$. Since the cohomology of $(\wedge V, \xi)$ vanishes as noted above, we have that $T$ is injective.
\end{proof}

\begin{prop}\label{harmonic prop} Let $\Omega_{L^2, \, \mathcal{C}^{\infty}}(X) = \{ f \in \Omega_{L^2}(X) \mid \, f|_{X_{reg}} \in \Omega_{\mathcal{C}^{\infty}}(X_{reg}) \}.$

If $X$ is compact, 
$$\ker (\Delta_d) \subset \Omega_{L^2, \, \mathcal{C}^{\infty}}(X) \textrm { and } \ker(\Delta_d) = \ker(d) \cap \ker(d^*).$$
\end{prop}
\begin{proof} As in Example \ref{Q_X example}, it is clear that $X = \cup_{\pi} \operatorname{im}(\pi)$ where $\pi$ runs over all trivial $\pi : F \to X$ with $F \subset \widetilde{X}$ closed.

Let $\pi : F \to X$ be among such $\pi$'s and $Y = \pi(F)$. We claim
$$\ker (\Delta_d|_{\Omega'(Y)}) \subset \cap_{s > 0} W^s(Y).$$
For that, first we note that $\|u\|_{(t)} < \infty$ for some negative integer $t$. By repeated use of Gårding's inequality (Proposition \ref{Garding}), it follows $\|u\|_{(t)} < \infty \Rightarrow \|u\|_{(t+2)} < \infty \Rightarrow \cdots.$ Thus, $\|u\|_{(s)} < \infty$ for each integer $s > 0$, establishing the claim.

Next, we shall show
$$\cap_{s > 0} W^s(Y) \subset \Omega_{L^2, \, \mathcal{C}^{\infty}}(Y).$$
By Corollary \ref{cor:integrable Fourier}, it is enough to show $\widehat{u}$ is integrable. By the Cauchy--Schwarz inequality, for $s > m/2$, we have:
$$\int |\widehat{u}| \, d\xi = \int |\widehat{u}| (1 + |\xi|^2)^{s/2} (1 + |\xi|^2)^{-s/2} \, d\xi \le c \|u\|_{(s)}$$
for some constant $c$. Hence, $\widehat{u}$ is indeed integrable. Moreover, $c$ and $\|u\|_{(s)}$ are bounded by constants independent of $F$. Thus, $u$ is square-integrable.


Finally, if $u$ is in $\Omega_{L^2}(X) \cap \Delta_d^{-1}(\Omega_{L^2}(X))$, we have:
$$(\Delta_d u, u) = \|du\|^2 + \|d^* u\|^2.$$
Thus, $\Delta_d u = 0$ if and only if $d u = 0 = d^* u$ on each compact subset and thus globally.
\end{proof}

In the next section (precisely, Example \ref{Poincare example}), we shall see that if $X$ is not smooth, then there can be a harmonic function on $X$ that is not continuous on $X$. Thus, in the above proposition, we get the smoothness of harmonic functions only on the regular locus.

We note the following simple fact from functional analysis.

\begin{lemma}\label{abstract elliptic} Let $H_1, H_2$ be Hilbert spaces and $E \subset H_1$ a dense subspace. Let $T : E \to H_2$ be a linear operator whose graph is closed in $H_1 \times H_2$. Also, let
$$G : \operatorname{im}(T) \to (\ker T)^{\bot} \cap E$$
be the inverse of $T : {(\ker T)^{\bot} \cap E} \to \im(T)$. ($G$ for the Green operator.)

Then the following are equivalent.
\begin{enumerate}
\item $\dim(\ker T) < \infty$ and $G$ is compact; i.e., it maps an open ball to a subset of a compact subset.
\item Each sequence in $E$ that is bounded in $\| \cdot \| + \| T \cdot \|$ has a convergent subsequence in $H_1$.
\end{enumerate}

Also, when the above equivalent conditions hold, we have, in particular,
$$\operatorname{im}(T) = \ker(T^*)^{\bot}.$$
\end{lemma}
\begin{proof} (ii) $\Rightarrow$ (i): First, (ii) implies that a closed ball in the kernel of $T$ is compact; thus, $\operatorname{ker}(T)$ has finite dimension by Riesz's lemma. To show $G$ is compact, we first show it is bounded. For that, suppose otherwise; then we can find a sequence $f_j$ in $\im(T)$ such that
$$1 = \| G f_j \| \ge j \|f_j\|.$$ Note $f_j \to 0$. Since $G f_j$ is bounded,
$$\| G f_j \| + \| TG f_j \| = \| G f_j \| + \| f_j \|$$
is bounded. Hence $G f_j$ has a convergent subsequence by (ii). Replacing $f_j$ by a subsequence, we can assume $G f_j$ is convergent with limit $u$. Since $T$ has closed graph, the graph of $G$ is closed. Thus, the limit $(0, u)$ of $(f_j, G f_j)$ is in the graph of $G$. Thus, $u = 0$. But $\|G f_j\| = 1$, $\|u\| = 1$, a contradiction. Now, if $f_j$ is a bounded sequence, then $G f_j$ is bounded and so like above, $f_j$ is bounded in $\| \cdot \| + \| T \cdot \|$. Thus, $f_j$ has a convergent subsequence.

(i) $\Rightarrow$ (ii): Let $u_j$ be a sequence in $E$ bounded in $\| \cdot \| + \| T \cdot \|$. Then let $v_j$ be the orthogonal projection of $u_j$ onto $\operatorname{ker}(T)^{\bot}$. Then $u_j - v_j$ is a bounded sequence in $\operatorname{ker}(T)$ and thus has a convergent subsequence. Also, since $T v_j = T u_j$ is bounded and $G$ is compact, $v_j = GT v_j$ has a convergent subsequence. Hence, $u_j = v_j + (u_j - v_j)$ has a convergent subsequence.

Finally, the ``also'' part holds by the closed range theorem, since $\im(T)$ is closed by (i).
\end{proof}

The next proposition is the key result in the classical Hodge theory, which continues to hold in our setup.

\begin{prop}\label{laplacian elliptic} Assume $X$ is compact. Let $E = \Delta_{d}^{-1}(\Omega_{L^2}(X)) \cap \Omega_{L^2}(X)$. Then the operator
$$\Delta_{d} : E \to \Omega_{L^2}(X)$$
satisfies the equivalent conditions in the above Lemma \ref{abstract elliptic}.

Also, the image of the above operator coincides with $\operatorname{ker}(\Delta_d)^{\bot}$.
\end{prop}
\begin{proof}[Proof (partly after the proof of \cite{Faltings} Theorem 3.1. (1))]

Let $T$ denote the operator in question. First we show $T$ has closed graph. Thus, suppose $(u_j, \Delta_d u_j) \to (u, f)$ in $\Omega_{L^2}(X)$. The convergence also holds in the topology on $\Omega'(X)$; namely, the weak-* topology. Since $\Delta_d$ is continuous in the topology on $\Omega'(X)$, we have $f = \Delta_d u$; i.e., $(u, f)$ is in the graph of $T$.

We shall now check the condition (ii) in the lemma. Thus, let $f_j$ be a sequence in the domain of $T$ that is bounded in $\|\cdot\| + \|T \cdot\|$. By the Cauchy--Schwarz inequality, we have: for some constant $c > 0$,
\begin{align*}
| \widehat{f_j}(\xi) - \widehat{f_j}(\eta) |^2 &\le \|f_j\|^2 \int_K |e^{- 2 \pi x \cdot \xi} - e^{- 2 \pi x \cdot \eta}|^2 \, dx. \\
&\le c |\xi - \eta|^2.
\end{align*} The above estimate then implies that $\widehat{f_j}$ is equicontinuous on each compact subset of $\mathbb{R}^n$. Since $\widehat{f_j}$ is also bounded in the sup norm on a compact subset by a similar argument, we can find, by Ascoli's theorem and by the diagonal argument, a subsequence of $f_j$ whose Fourier transform converges in the sup norm on each compact subset of $\mathbb{R}^m$. We shall denote that subsequence again by $f_j$.

Now, by the Plancherel theorem; i.e., $\widehat{\cdot}$ is unitary, we have:
\begin{align*}
\| f_j - f_k \|^2 &= \| \widehat{f_j} - \widehat{f_k} \|^2 \\
&\le \int_{|\xi| \le R} | \widehat{f_j} - \widehat{f_k} |^2 \, d \xi + (1 + R^2)^{-s} \| f_j - f_k \|_{(s)}.
\end{align*}
By Gårding's inequality (Proposition \ref{Garding}), the sequence $f_j$ is bounded in the Sobolev norm $\| \cdot \|_{(s)}$ for $s = 2$. Given $\epsilon > 0$, first choose $R > 0$ so that the second term is less than $\epsilon/2$, independent of $j, k$. Then since $\widehat{f_j}$ is Cauchy on each compact subset, for large enough $j, k$, the first term is less than $\epsilon/2$, finishing the proof of the claim.

For the final assertion, on one hand, $\operatorname{im}(T)^{\bot} \subset \operatorname{ker}(T)$, since for $u$ in $\operatorname{im}(T)^{\bot}$, we have:
$$(Tu, f) = (u, Tf) = 0$$
for $f$ in $\Omega_{cpt}(X_{reg})$. Similarly, $\operatorname{ker}(T) = \operatorname{im}(T^*)^{\bot} \subset \operatorname{ker}(T^*)$ as $T, T^*$ agree on a dense subspace. Thus, we have: $\operatorname{im}(T) \subset \operatorname{ker}(T^*)^{\bot} \subset \operatorname{ker}(T)^{\bot}$. Since $\operatorname{im}(T)$ is closed by the lemma; thus, $\operatorname{im}(T) = \operatorname{im}(T)^{\bot\bot}$, we conclude the claimed equality.
\end{proof}

\begin{remark} We do not expect the above proposition to hold without compactness. This is because the proposition would then imply, in particular, the kernel of $\Delta_d$ has finite dimension. But, as we will see below, the kernel of $\Delta_d$ coincides with the $L^2_{loc}$-cohomology of $X$ and the latter may not have finite dimension without compactness.

The above proof establishes Rellich's lemma in our setup (cf., \cite[Theorem 10.1.27.]{hormander2}).
\end{remark}

Let $\Delta_{L^2}$ denote the operator in the above Proposition \ref{laplacian elliptic}. Since $\operatorname{ker}(\Delta_d)^{\bot} = \im(\Delta_{L^2})$, we have the orthogonal decomposition known as the Hodge decomposition:
$$\Omega_{L^2}(X) = \operatorname{ker}(\Delta_d) \oplus \operatorname{im}(\Delta_{L^2})$$
when $X$ is compact.

For not-necessarily compact $X$, a harmonic form is not necessarily square-integrable. But we still have the following lemma.

We define $\Omega_{L^{2, \, loc}}$ to be the sheaf on $X$ where a section over an open set $U$ is a current on $U$ such that $\psi \, u$ is in $\Omega_{L^2}(U)$ for each $\psi$ in $\mathcal{O}_{cpt}(U)$.

\begin{lemma}\label{Hodge decomposition} For each $u$ in $\Omega_{L^{2, \, loc}}(X)$, there exists a unique $u_0$ in $\operatorname{ker}(\Delta_d)$ and $f$ in $\Omega_{L^{2, \, loc}}(X)$ such that
$$u = u_0 + \Delta_d f.$$
If $u$ is $d$-closed, then the above $f$ is also $d$-closed.
\end{lemma}
\begin{proof}
Let $U_i$ be an open cover of $X$ such that each closure $K_i = \overline{U_i}$ is a compact PA-space (possible since $X$ is locally compact). Let $u_i = u|_{K_i}$. Then, by the Hodge decomposition, we have
$$u_i = u_{i, 0} + \Delta|_{K_i} G_i (u_i - u_{i, 0})$$
where $u_{i, 0}$ is the orthogonal projection of $u|_{K_i}$ onto $\operatorname{ker}(\Delta|_{K_j})$. Let $K = K_i \cap K_j$. Then $u|_K - u_{i, 0}|_K$ is in $\im(\Delta|_K)$ and so $u_{i, 0}|_K$ is exactly the harmonic part of $u|_K$. It follows $u_{i, 0}, u_{j, 0}$ agree on $K = K_i \cap K_j$; a fortiori, on $U_i \cap U_j$. Hence, we have a $u_0$ in $\operatorname{ker}(\Delta_d)$ such that $u_0|_{U_i} = u_{i, 0}|_{U_i}$.

Next, we note that, on each compact subspace $K \subset X$, the Laplace equation
$$\Delta_d f = u - u_0$$
has a unique solution for $f$ in $\operatorname{ker}(\Delta|_K)^{\bot}$, since if $g$ is another such solution, $f - g$ is in $\ker(\Delta|_K) \cap \operatorname{ker}(\Delta|_K)^{\bot} = 0$. Now, for $f_i = G_i (u_i - u_{i, 0})$ and $K \subset K_i$, we have $f_i|_K$ is in $\im(\Delta|_{K, L^2}) = \ker(\Delta|_{K, , L^2})^{\bot}$ and so $f_i|_K$ is a unique solution in the above sense. Thus, $f_i = f_j$ on $K_i \cap K_j$; a fortiori, on $U_i \cap U_j$. Hence, they glue to a form $f$ in $\Omega_{L^{2, \, loc}}(X)$.

Finally, if $X$ is compact, then let $f = G(u - u_0)$. Also,  we write $f = \Delta_{L^2}g$ for some $g$ in the domain of $\Delta_{L^2}$. By Gårding's inequality, we see $d f$ and $dg$ are both in $\Omega_{L^2}(X)$. So,
$$df = d \Delta_d g = \Delta_d dg \in \im(\Delta_{L^2}) = \ker(\Delta_{L^2})^{\bot}.$$
Thus,
$$df = G \Delta_{L^2} df = G d \Delta_{L^2} G(u - u_0) = Gd u$$
since $u_0$ is harmonic and thus closed (Proposition \ref{harmonic prop}). Hence, if $du = 0$, then $df = 0$.
\end{proof}

We first prove the Hodge theorem for the $L^2_{loc}$-version of cohomology. Note $\Omega_{L^{2, \, loc}}$ is not a complex, since the differential $d$ is not well-defined. Thus, we modify it as follows: we let
$$\Omega_{L^{2, \, loc}}^{*, p} = d^{-1}(\Omega^{p+1}_{L^{2, \, loc}}) \cap \Omega^{p}_{L^{2, \, loc}}.$$
Then $(\Omega_{L^{2, \, loc}}^{*}, d)$ is a complex and we let $\operatorname{H}^p_{L^{2, \, loc}}(X)$ be the cohomology of $\Gamma(X, \Omega_{L^{2, \, loc}}^{*}).$


With this cohomology, we have the following (the singular version will be considered later):

\begin{theorem}\label{L^2 Hodge} $$\operatorname{ker}(\Delta_d) \to \operatorname{H}_{L^{2, \, loc}}^*(X), \, u \mapsto [u]$$
is a well-defined isomorphism.
\end{theorem}
\begin{proof} The map is well-defined by Proposition \ref{harmonic prop}, since a harmonic form is closed and is locally $L^2$.

(injectivity) Suppose $[u] = 0$ for some $u$ in $\ker(\Delta_d)$. By Proposition \ref{harmonic prop}, $u$ is in $\Omega_{L^{2, \, loc}}(X)$. To show $u = 0$, it is enough to show $u = 0$ on each compact subset of $X$. Thus, assume $X$ is compact. We write $u = df$ for some $f$ in $\Omega_{L^2}(X)$. Then, for $g$ in $\operatorname{ker}(d^*) \cap \Omega_{L^2}(X) = \operatorname{ker}(d^*_{L^2})$, we have:
$$(u, g) = (df, g) = (f, d^* g) = 0.$$
That is, $u$ is orthogonal to $\operatorname{ker}(d^*_{L^2})$. On the other hand, $u$ is in $\operatorname{ker}(d^*)$ by Proposition \ref{harmonic prop}. Hence, $u = 0$.

(surjectivity) By the above Lemma \ref{Hodge decomposition}, we can write $u = u_0 + \Delta_d f$ where $f$ is closed since $u$ is. Since $\Delta_d f = d d^* f$, it follows that $[u] = [u_0]$.
\end{proof}

An immediate consequence is the following well-known

\begin{corollary}[Poincar\'e duality]\label{Poincare duality} If $X$ is compact, the natural pairing on $\operatorname{ker}(\Delta_d)$, $$(f, g) \mapsto \int_{X_{reg}} f \wedge g$$
induces the perfect pairing
$$\operatorname{H}^p_{L^2}(X) \times \operatorname{H}^{n-p}_{L^2}(X) \to \mathbb{C}.$$
\end{corollary}
\begin{proof} Because of the finite-dimensionality, it is enough to show the pairing on $\operatorname{ker}(\Delta_d)$ is nondegenerate. Thus, given a harmonic $p$-form $f$, suppose $\int_{X_{reg}} f \wedge g = 0$ for each harmonic $(n-p)$-form $g$. Since $* \Delta_d = \Delta_d *$, $*f$ is harmonic.
Taking $g = *f$, we have $0 = \int_{X_{reg}} f \wedge g = \|f\|^2$ and so $f = 0$.
\end{proof}

We know that the Poincar\'e duality in the usual formulation can fail in general for a compact manifold with boundary or a singular complex algebraic variety; see Example \ref{Poincare example}. In particular, in general, $\operatorname{H}^*_{L^2}(X) \not\simeq \operatorname{H}^*(X; \mathbb{C})$ even when $X$ is compact.

The next is the Hodge theorem for singular cohomology.

\begin{theorem}\label{singular Hodge theorem} We have natural degree-preserving linear maps
$$\alpha : \operatorname{ker}(\Delta_d) \to \operatorname{H}^*(X;\mathbb{C})$$
and 
$$\beta : \operatorname{H}^*(X; \mathbb{C}) \to \operatorname{ker}(\Delta_d)$$
such that $\alpha \circ \beta$ is the identity.
\end{theorem}
\begin{proof} Let $Z^p = \operatorname{ker}(d|_{\Omega'^{\, p}})$. Then we have the quotient map
$$Z^p \to Z^p/d({\Omega'^{\, p-1}}) = \operatorname{H}^p(X; \mathbb{C}_{Po}).$$
Also, we have the sheaf morphism
$$\mathbb{C}_{Po} \to \mathbb{C}$$
as follows. We have the integration for each connected open subset $U \subset X$,
$$\pi_* : \operatorname{H}^n_{cpt}(U) \to \mathbb{C}, \, [\omega] \mapsto \int_{X_{reg}} \omega$$
since $\int_{X_{reg}} d = 0$ by Lemma \ref{lem:integration map}. It is a continuous linear functional; i.e., an element of $\operatorname{H}^n_{cpt}(U)^*$. Hence, we get $\mathbb{C}_{Po}(U) \to \mathbb{C}(U) = \mathbb{C}$ by mapping $\pi_*$ to $1$ and the complementary subspace to it to zero. Thus, we get the presheaf morphism $\operatorname{H}^n_{cpt}(U)^* \to \mathbb{C}$, which gives $\mathbb{C}_{Po}(U) \to \mathbb{C}(U) = \mathbb{C}$.

By Proposition \ref{harmonic prop}, $\ker(\Delta_d|_{\Omega'^{\, p}}) \subset Z^p$ and so we get the claimed map
$$\alpha : \ker(\Delta_d|_{\Omega'^{\, p}}) \overset{u \mapsto [u]}\to \operatorname{H}^p(X; \mathbb{C}_{Po}) \to \operatorname{H}^p(X; \mathbb{C}).$$

To construct $\beta$, first note that the Poincar\'e lemma holds for $\Omega_{PA, \, \mathcal{C}^0} = \Omega_{PA} \cap \Omega_{\mathcal{C}^0}$, the sheaf of continuous PA-forms. Thus, $\operatorname{H}^*(X;\mathbb{C})$ can be identified with the cohomology of the complex $(\Omega_{PA, \, \mathcal{C}^0}(X) \cap d^{-1}(\Omega_{PA, \, \mathcal{C}^0}(X)), d)$. Then we can define $\beta$ by $\beta([u]) = u_0$, the harmonic part of $u$. Finally, it is clear that $\alpha \circ \beta = \operatorname{id}$.


\end{proof}

\begin{corollary} If $X$ is compact, then $\operatorname{H}^*(X; \mathbb{C})$ has finite dimension.
\end{corollary}
\begin{proof} This follows from (i) in Lemma \ref{abstract elliptic}.
\end{proof}

\section{Further remarks}\label{sec:Further remarks}

This section is an assortment of various results that we did not need earlier but are still of some interest.

\subsection{Examples}

Here, we shall illustrate and confirm some of the results of this paper.

\begin{example}[cf. \cite{Kirwan} \S 1.1.]\label{Poincare example} Let
$$X = V(s_1s_2) = Y_1 \cup Y_2 \subset \mathbb{P}^2$$
over complex numbers where $s_1, s_2, s_3$ are homogeneous coordinates on $\mathbb{P}^2$ and $Y_j = V(s_j) \simeq \mathbb{P}^1$. Note $X - X_{reg} = Y_1 \cap Y_2$ is exactly the point $0 = [0 : 0 : 1]$ since the local ring at that point is not an integral domain (alternatively, the tangent space at that point has dimension 2).

Now, since $X$ is connected, $\operatorname{H}^0(X) = \operatorname{H}^0(X; \mathbb{C})$ has dimension $1$, while $\operatorname{H}^2(X) \simeq \bigoplus_j \operatorname{H}^2(Y_j)$ has dimension $2$ by the Mayer--Vietoris sequence \cite[Ch. 19., \S 3.]{May}:
$$\cdots \to \operatorname{H}^1(Y_1 \cap Y_2) \to \operatorname{H}^2(X) \to \oplus_{j=1}^2 \operatorname{H}^2(Y_j) \to \operatorname{H}^2(Y_1 \cap Y_2) \to 0,$$
where the sequence holds since we can assume $X, Y_1, Y_2$ are CW-complexes in such a way $Y_j, Y_1 \cap Y_2$ are the subcomplexes. In particular, the Poincar\'e duality fails. Also, $\operatorname{H}^1(X) = 0$ since there is no cell of odd dimension in $X$.

Next, we compute $\operatorname{H}^p(X; \mathbb C)$ as the cohomology of PA-forms for $p = 0, 2$ (the case $p = 1$ is omitted). For $p = 0$, we have that $df = 0$ for $f$ in $\Omega^0_{PA}(X)$ implies $f$ is constant since $X$ is connected. Thus, $\operatorname{H}^0(X)$ has dimension one. For $p = 2$, let $\omega_j$ be the standard volume forms on $Y_j \simeq \mathbb{P}^1$. By Corollary \ref{cor:PA-form Mayer}, we have $\Omega^2_{PA}(X) \simeq \oplus_{j=1}^2 \Omega_{PA}^2(Y_j)$. Thus, we can extend $\omega_j$ on $Y_j$ to $X$ such that $\omega_j = 0$ on $Y_{3 - j}$. Note at the point $0$, $\omega_1$ is zero since there is no nonzero higher form at a point.

By Corollary \ref{cor:PA-form Mayer}, we have the Mayer--Vietoris sequence for the cohomology of PA-forms as well. Thus, we have $\operatorname{H}^2(X) \simeq \bigoplus_j \operatorname{H}^2(Y_j)$ and $[\omega_j|_{Y_j}]$ spans $\operatorname{H}^2(Y_j) \simeq \mathbb{C}$. Thus, $[\omega_1], [\omega_2]$ form a basis of $\operatorname{H}^2(X)$.

Finally, we shall compute $\operatorname{ker}(\Delta|_{\Omega'^{\, p}})$ and the map from it to the cohomology. Let $\sigma = \omega_1 + \omega_1$. Then $*\sigma = 1$. Then
$$\Delta \sigma = d d^* \sigma + d^* d\sigma = 0$$
since $d \sigma = 0$ and since $d(* \sigma) = d(1) = 0.$ Thus, $\sigma$ is harmonic; i.e., in $\operatorname{ker}(\Delta)$.




Define the function $f$ on $X$ by $f = 1$ on $Y$ and $f = 0$ on $X - Y$. Then $f$ is in $L^2(X) \subset \Omega'^{\, 0}(X)$. Now, since $d^* f = 0$ for dimension reason, we have $\Delta f = d^* d f$. But $df = 0$ since
$$\langle df, \psi \rangle = -(f, d\psi) = -\int_{Y_{reg}} d \psi = 0$$
by Lemma \ref{lem:integration map}. Hence, $f$ is harmonic.

\end{example}

\begin{example}\label{metric example}
Take $X$ to be the closed upper half plane
$\{(x,y) \mid 0\le y\}$ in $\mathbb{R}^2$. Then $X_{reg}\subset X$ is the open upper half plane. Then the symmetric positive definite matrix  
\[
\omega=
\begin{bmatrix}
    x+1 & \frac{1}{\sqrt{x}} \\
    \frac{1}{\sqrt{x}} & \frac{1}{x}
\end{bmatrix}.
\]
gives a Riemannian metric on the manifold $X_{reg}$ by
$(v,w) = v^T \omega \, w$ when we identify the tangent vectors with vectors in $\mathbb{R}^2$. As $\det(\omega)=1$, its volume form in the standard coordinates is $\mu=dx\wedge dy$. However, it is clear that we can also find a metric (e.g., the identity matrix) that gives rise to the same volume form $\mu$.
\end{example}

\begin{example}[cf. \cite{Cruz} Example 2.1.]\label{metric example 2} Let $X$ be a one-dimensional compact PA-space such that there is a diffeomorphism $\psi : (0, \infty) \overset{\sim}\to X_{reg}$. Then equip $X_{reg}$ with the metric induced from it; explicitly, $|\omega| := |\psi^*(\omega)|$ on $X_{reg}$ for each $\omega$ in $\Omega(X)$.

Now, let $f_r$ be a one-form in $\Omega'(X)$ such that $\psi^*(f_r|_{X_{reg}}) = (1+x)^r \, dx$. If $f_r$ is in $\Omega_{L^2}(X)$, then $r < -\frac{1}{2}$ and conversely. If $du = f_r$, then $\psi^*u' = (1+x)^r$; that is, $\psi^* u = \frac{(1+x)^{r+1}}{r+1}$ up to some constant. But if $-1 < r < -\frac{1}{2}$, then there is no such $u$ in $\Omega_{L^2}(X)$ and so $[f_r] \ne 0$ in $\operatorname{H}^*_{L^2}(X)$. Since $\{ [f_r] \mid -1 < r < -\frac{1}{2} \}$ is linearly independent and is infinite, we have that the dimension of $\operatorname{H}^1_{L^2}(X)$ is infinite. In particular, the Hodge theorem (Theorem \ref{L^2 Hodge}) fails for $X$ here. In other words, for that theorem to hold, we cannot just choose some arbitrary metric on $X_{reg}$.
\end{example}

\subsection{The category of compact PA-spaces}

\begin{lemma}[\cite{Kontsevich} \S 8.1. Lemma 8.] Given compact PA-spaces $X, Y$, there is a natural bijection
$$\operatorname{Hom}_{PA}(X, Y) \simeq \Hom_{\mathbb{C}-alg}(\mathcal{O}(Y), \mathcal{O}(X))$$
given by $f \mapsto f^*$.
\end{lemma}
\begin{proof} Let $\varphi : \mathcal{O}(Y) \to \mathcal{O}(X)$ be an algebra homomorphism. We shall define $f$ so that $f^* = \varphi$. We can assume $Y \subset \mathbb{R}^n$ and then let $y_i$ be the coordinates on $\mathbb{R}^n$. If $\varphi = f^*$, then for each point $p$ on $X$,
$$\varphi(y_i)(p) = f^*(y_i)(p) = f_i(p)$$
where $f_i = y_i \circ f$. Thus, we shall define $f = (f_1, \cdots, f_n) : X \to \mathbb{R}^n$ by $f_i = \varphi(y_i)$. Note $\varphi(g) = f^*(g)$ for each polynomial $g$ in $y_i$'s. By the Stone--Weierstrass theorem, the ring of polynomial functions on $Y$ is dense in $\mathcal{O}(Y)$ with respect to the sup norm. Thus, if $\varphi$ is continuous, then $\varphi = f^*$.

To see the continuity of $\varphi$, we first note that if $\omega$ is a character on $\mathcal{O}(Y)$; i.e., an algebra homomorphism $\mathcal{O}(Y) \to \mathbb{C}$, then $\omega$ is continuous with the operator norm $\|\omega\| = 1$. Indeed, the kernel of $\omega$ is a maximal ideal; thus, is closed and that implies $\omega$ is continuous (\cite[Theorem 1.18.]{Rudin}). Then $|\omega(g)||\omega(h)| = |\omega(gh)| \le \|\omega\|$ for $\|g\|, \|h\| \le 1$ and thus $\|\omega\|^2 \le \|\omega\|$ or $\|\omega\| = 1$, since $\omega(1) = 1$. Now, taking $\omega$ to be the composition
$$\mathcal{O}(Y) \overset{\varphi}\to \mathcal{O}(X) \overset{p^*}\to \mathbb{C}$$
for a point $p$ on $X$ and $p^*$ the evaluation at $p$, we have $|\varphi(g)(p)| \le \|g\|$ for each $g$ in $\mathcal{O}(Y)$. Thus, $\|\varphi(g)\| \le \|g\|$, finishing the proof of the continuity.

It remains to show the image of $f$ lies in $Y$. Suppose $Y = \{ g_1 \le a_i, \cdots, g_r \le a_r \}$, where there is no strict inequality since $Y$ is compact. Since $Y$ is compact, each $g_i$ is bounded below. Thus, by replacing $f_i, a_i$ by $f_i - c_i, a_i - c_i$ for some constants $c_i$, we can assume $g_i \ge 0$ on $Y$. Also, without loss of generality, we can assume $\varphi(g_i)$ is real-valued. Then
$$g_i \circ f = \varphi(g_i) \le |\varphi(g_i)|\le \|\varphi\| \sup_Y |g_i|.$$
By the first part, $\| \varphi \| \le 1$. Also, $|g_i| = g_i \le b_i$ on $Y$. Thus, the image of $f$ is contained in $Y$.
\end{proof}

As a consequence, we have:

\begin{prop} The category of compact PA-spaces embeds (i.e., there is a fully faithful functor) into the category of affine schemes over real numbers.
\end{prop}

While the above seems somewhat interesting, we are not sure if it is useful, since affine schemes corresponding to compact PA-spaces are generally non-Noetherian.

\subsection{dg-algebra of PA-forms}

We also note:

\begin{prop} $\Omega_{X, PA}$ is a differential-graded sheaf of algebras.

Precisely, $\Omega_{PA}(X)$ is a dg-subalgebra of $\varinjlim_{U} \Omega_{\mathcal{C}^{\infty}}(U)$ where the inclusion into inductive limit is given in Proposition \ref{PA approximation}.
\end{prop}
\begin{proof} Let $\alpha = \int_{\gamma_1} \omega_1, \, \beta = \int_{\gamma_2} \omega_2$, where $\gamma_i : X \to C(Y_i)$. Let $Y = Y_1 \times Y_2$ and $p_i : Y \to Y_i$ the projections. Then, by Fubini's theorem, up to signs,
$$\alpha \wedge \beta = \int_{p_2^* \gamma_2} \left(\int_{p_1^* \gamma_1} \omega_1 \right) \wedge \omega_2 = \int_{\gamma'} \omega_1 \wedge \omega_2$$
where $\gamma' = \gamma_1 \times_X \gamma_2$.
\end{proof}


\subsection{Derivatives of PA-forms}\label{sec:differential operator}

We record some results relating to the problem of differentiating PA-forms. In short, it's not easy and we only have some partial results. To get a sense of the issue, first consider a simple example

\begin{example} Let $f(x) = \sqrt{x}$ on $X = [0, \infty)$. Then $f' = \frac{1}{2 \sqrt{x}}$ is not continuous on $X$ (not even defined).

Similarly, let $f(x) = \sqrt{|x|}$ on $X = \mathbb{R}$. Then $f$ is continuously differentiable on $X - 0$ but not at $0$.

In both cases, $x f'$ can however be defined and is continuous on $X$.
\end{example}



As usual, we define the interior product $i_{\xi}$ by $(i_{\xi} \omega)(-) = \omega(\xi, -)$. We then define the \emph{Lie derivative $L_{\xi}$ along $\xi$} by
$$L_{\xi} = d \, i_{\xi} + i_{\xi} \, d$$
(cf. \cite[Proposition 2.25. (d)]{Warner}).

This operator is not necessarily defined for all the local sections of $\Omega$ nor of $\Omega_{PA}$. The next lemma gives an instance when a Lie derivative of a PA-form is a PA-form.

\begin{lemma}[Lie derivative under the integral sign]\label{differentiation under the integral sign} Given a PA-form $\alpha = \int_{\gamma} \omega$ with $\pi : Y \to X$, if a vector field $\xi$ admits a lift $\widetilde{\xi}$ to $Y$; i.e., $d\pi \circ \widetilde{\xi} = \xi \circ \pi$ and if $i_{\widetilde{\xi}} \, \omega$ is in $\Omega(Y)$, then
$$i_{\xi} \alpha = (-1)^r \int_{\gamma} i_{\widetilde{\xi}} \, \omega.$$
As a corollary, if $i_{\widetilde{\xi}} \, d\omega$ is also in $\Omega(Y)$, then
$$L_{\xi} \, \alpha = \int_{\gamma} L_{\widetilde \xi} \,\omega$$
and thus is in $\Omega_{PA}(Y)$.
\end{lemma}
\begin{proof} For the first assertion, it suffices to show that on some open dense definable subset. Thus, we can assume $\pi$ is locally trivial. Then, by the sheaf property, it is enough to show the case when $\pi$ is trivial.

If $\pi$ is trivial, then $\omega$ is a $\mathbb{Z}$-linear combination of forms of the form $\eta \wedge \pi^* \sigma$, $\eta$ an $s$-form with $s \le r$ and we only need to consider the term with $s = r$, since otherwise the both sides vanish. By the projection formula and the Leibniz rule, we have:
$$i_{\xi} \int_{\gamma} \eta \wedge \pi^* \sigma = i_{\xi} \left( \left(\int_{\gamma} \eta \right) \sigma \right) = \left(\int_{\gamma} \eta \right) i_{\xi} \sigma = \int_{\gamma} \eta \wedge \pi^* i_{\xi} \sigma.$$
We have $i_{\widetilde{\xi}}(\pi^* \sigma) = \pi^*(i_{\xi} \sigma)$ and so, by the Leibniz rule and the trivial fact $\int_{\gamma} i_{\widetilde{\xi}} \eta = 0$, it equals:
$$(-1)^r \int_{\gamma} i_{\widetilde{\xi}} (\eta \wedge \pi^* \sigma).$$
The formula for the Lie derivative then follows from Proposition \ref{FTC}.
\end{proof}

We record

\begin{conjecture} If $f$ is a $PA$-form on $X$, then $P f$ is in $\Omega(X)$ for some differential operator $P$.
\end{conjecture}

Roughly, the conjecture should hold since a PA-form is obtained by integrating forms in $\Omega(X)$. If the above conjecture is true, then in particular we can conclude that the the exponential function $z \mapsto e^z$ is not in $\Omega_{PA}(X)$ with $X = \mathbb{C}$.

\footnotesize{
\bibliographystyle{alpha}
\bibliography{bib}
}

\end{document}